\documentclass[12pt, reqno]{amsart}
\usepackage{appendix}
\usepackage[a4paper, centering]{geometry}
\usepackage{xcolor}
\usepackage[utf8]{inputenc}
\usepackage[colorlinks=true,hyperindex,pagebackref,linktocpage=true]{hyperref}
\usepackage{cleveref}
\usepackage{mathtools}
\usepackage{amsfonts}
\usepackage{amssymb}
\usepackage{tikz-cd}
\usepackage[T1]{fontenc}

\usepackage{amssymb}    
\usepackage{graphicx}   
\usepackage{scalerel}   

\usepackage{amsmath}

\hypersetup{
	colorlinks,
	linkcolor={violet!60!black},
	citecolor={cyan!60!black},
	urlcolor={orange!60!black}
}
\usepackage{stmaryrd}
\usepackage{mathrsfs}  
\usepackage{varwidth}
\theoremstyle{plain}
\newtheorem{theorem}{Theorem}[section]
\newtheorem{proposition}[theorem]{Proposition}

\newtheorem{lemma}[theorem]{Lemma}
\newtheorem{corollary}[theorem]{Corollary}

\newtheorem{question}[theorem]{Question}

\theoremstyle{definition}
\newtheorem{definition}[theorem]{Definition}
\newtheorem{example}[theorem]{Example}
\newtheorem{remark}[theorem]{Remark}

\newcommand{\nc}{\newcommand}
\nc{\on}{\operatorname}

\nc{\Q}{\mathbb{Q}}
\nc{\Z}{\mathbb{Z}}
\nc{\cl}{\mathrm{cl}}

\newcommand{\pitch}{{\mathbin{\raisebox{1.5 ex}{\reflectbox{\scalebox{1}[-1]{{$\mathbf{\pitchfork}$}}}}}}}

\nc{\fraka}{{\mathfrak a}} \nc{\bba}{{\mathbf a}}
\nc{\frakb}{{\mathfrak b}}
\nc{\frakc}{{\mathfrak c}}
\nc{\frakd}{{\mathfrak d}}
\nc{\frake}{{\mathfrak e}}
\nc{\frakf}{{\mathfrak f}}
\nc{\frakg}{{\mathfrak g}}
\nc{\frakh}{{\mathfrak h}}
\nc{\fraki}{{\mathfrak i}}
\nc{\frakj}{{\mathfrak j}}
\nc{\frakk}{{\mathfrak k}}
\nc{\frakl}{{\mathfrak l}}
\nc{\frakm}{{\mathfrak m}}
\nc{\frakn}{{\mathfrak n}}
\nc{\frako}{{\mathfrak o}}
\nc{\frakp}{{\mathfrak p}}
\nc{\frakq}{{\mathfrak q}}
\nc{\frakr}{{\mathfrak r}}
\nc{\fraks}{{\mathfrak s}}
\nc{\frakt}{{\mathfrak t}}
\nc{\fraku}{{\mathfrak u}}
\nc{\frakv}{{\mathfrak v}}
\nc{\frakw}{{\mathfrak w}}
\nc{\frakx}{{\mathfrak x}}
\nc{\fraky}{{\mathfrak y}}
\nc{\frakz}{{\mathfrak z}}
\nc{\frakA}{{\mathfrak A}}
\nc{\frakB}{{\mathfrak B}}
\nc{\frakC}{{\mathfrak C}}
\nc{\frakD}{{\mathfrak D}}
\nc{\frakE}{{\mathfrak E}}
\nc{\frakF}{{\mathfrak F}}
\nc{\frakG}{{\mathfrak G}}
\nc{\frakH}{{\mathfrak H}}
\nc{\frakI}{{\mathfrak I}}
\nc{\frakJ}{{\mathfrak J}}
\nc{\frakK}{{\mathfrak K}}
\nc{\frakL}{{\mathfrak L}}
\nc{\frakM}{{\mathfrak M}}
\nc{\frakN}{{\mathfrak N}}
\nc{\frakO}{{\mathfrak O}}
\nc{\frakP}{{\mathfrak P}}
\nc{\frakQ}{{\mathfrak Q}}
\nc{\frakR}{{\mathfrak R}}
\nc{\frakS}{{\mathfrak S}}
\nc{\frakT}{{\mathfrak T}}
\nc{\frakU}{{\mathfrak U}}
\nc{\frakV}{{\mathfrak V}}
\nc{\frakW}{{\mathfrak W}}
\nc{\frakX}{{\mathfrak X}}
\nc{\frakY}{{\mathfrak Y}}
\nc{\frakZ}{{\mathfrak Z}}
\nc{\bbA}{{\mathbb A}}
\nc{\bbB}{{\mathbb B}}
\nc{\bbC}{{\mathbb C}}
\nc{\bbD}{{\mathbb D}}
\nc{\bbE}{{\mathbb E}}
\nc{\bbF}{{\mathbb F}} \nc{\bbf}{{\mathbf f}}
\nc{\bbG}{{\mathbb G}}
\nc{\bbH}{{\mathbb H}}
\nc{\bbI}{{\mathbb I}}
\nc{\bbJ}{{\mathbb J}}
\nc{\bbK}{{\mathbb K}}
\nc{\bbL}{{\mathbb L}}
\nc{\bbM}{{\mathbb M}}
\nc{\bbN}{{\mathbb N}}
\nc{\bbO}{{\mathbb O}}
\nc{\bbP}{{\mathbb P}}
\nc{\bbQ}{{\mathbb Q}}
\nc{\bbR}{{\mathbb R}}
\nc{\bbS}{{\mathbb S}}
\nc{\bbT}{{\mathbb T}}
\nc{\bbU}{{\mathbb U}}
\nc{\bbV}{{\mathbb V}}
\nc{\bbW}{{\mathbb W}}
\nc{\bbX}{{\mathbb X}}
\nc{\bbY}{{\mathbb Y}}
\nc{\bbZ}{{\mathbb Z}}
\nc{\calA}{{\mathcal A}}
\nc{\calB}{{\mathcal B}}
\nc{\calC}{{\mathcal C}}
\nc{\calD}{{\mathcal D}}
\nc{\calE}{{\mathcal E}}
\nc{\calF}{{\mathcal F}}
\nc{\calG}{{\mathcal G}}
\nc{\calH}{{\mathcal H}}
\nc{\calI}{{\mathcal I}}
\nc{\calJ}{{\mathcal J}}
\nc{\calK}{{\mathcal K}}
\nc{\calL}{{\mathcal L}}
\nc{\calM}{{\mathcal M}}
\nc{\calN}{{\mathcal N}}
\nc{\calO}{{\mathcal O}}
\nc{\calP}{{\mathcal P}}
\nc{\calQ}{{\mathcal Q}}
\nc{\calR}{{\mathcal R}}
\nc{\calS}{{\mathcal S}}
\nc{\calT}{{\mathcal T}}
\nc{\calU}{{\mathcal U}}
\nc{\calV}{{\mathcal V}}
\nc{\calW}{{\mathcal W}}
\nc{\calX}{{\mathcal X}}
\nc{\calY}{{\mathcal Y}}
\nc{\calZ}{{\mathcal Z}}

\nc{\scrA}{{\mathscr A}}
\nc{\scrB}{{\mathscr B}}
\nc{\scrC}{{\mathscr C}}
\nc{\scrD}{{\mathscr D}}
\nc{\scrE}{{\mathscr E}}
\nc{\scrF}{{\mathscr F}}
\nc{\scrG}{{\mathscr G}}
\nc{\scrH}{{\mathscr H}}
\nc{\scrI}{{\mathscr J}}
\nc{\scrJ}{{\mathscr I}}
\nc{\scrK}{{\mathscr K}}
\nc{\scrL}{{\mathscr L}}
\nc{\scrM}{{\mathscr M}}
\nc{\scrN}{{\mathscr N}}
\nc{\scrO}{{\mathscr O}}
\nc{\scrP}{{\mathscr P}}
\nc{\scrQ}{{\mathscr Q}}
\nc{\scrR}{{\mathscr R}}
\nc{\D}{{\on{D}}}
\nc{\Div}{{\on{Div}}}
\nc{\Perv}{{\on{Perv}}}

\nc{\bnu}{{\bar{ \nu}}}
\nc{\olO}{\bar{\calO}}

\nc{\al}{{\alpha}} 
\nc{\be}{{\beta}}
\nc{\ga}{{\gamma}} \nc{\Ga}{{\Gamma}}
\nc{\hGa}{\hat{\Gamma}}
\nc{\ve}{{\varepsilon}} 
\nc{\la}{{\lambda}} \nc{\La}{{\Lambda}}
\nc{\om}{\omega} \nc{\Om}{\Omega} 
\nc{\sig}{{\sigma}} \nc{\Sig}{{\Sigma}}
\nc{\dR}{{\mathrm{dR}}}
\nc{\Perf}{{\mathrm{Perf}}}
\nc{\PSch}{{\mathrm{PSch}}}
\nc{\Gm}{{\mathbb{G}_m}}
\nc{\colim}{{\on{colim}}}
\nc{\et}{\mathrm{\acute{e}t}}
\nc{\mer}{{\on{mer}}}
\nc{\sht}{{\on{sht}}}
\nc{\sss}{{\on{ss}}}
\nc{\Sets}{{\on{Sets}}}

\nc{\vvb}{{\on{Vect}^{\calO^\sharp}_{\on{v}}}}
\nc{\Vect}{{\on{Vect}}}
\nc{\onv}{{\on{v}}}

\nc{\DM}{{\frakD\frakM}}
\nc{\DMG}{{\frakD\frakM_{\calG}}}
\nc{\IC}{{\mathfrak{B}}}
\nc{\ICG}{{\frakB(G)}}
\nc{\SHT}{{\calS\calH\calT}}
\nc{\VEC}{{\on{Bun}}_{\on{FF}}^{\on{mer}}}
\nc{\Catex}{{\on{Cat}^{\otimes, \on{ex}}_1}}
\nc{\Cat}{{\on{Cat}^\otimes_1}}
\nc{\Fil}{\calF{il}}
\nc{\Dm}{{\on{DM}}}
\nc{\Sht}{\on{Sht}}
\nc{\WSht}{\on{WSht}}
\nc{\Isoc}{\on{Isoc}}

\nc{\BC}{\calB\calC}
\nc{\PFS}{\bbP_{{\on{FS}}}}
\def\preceqdot{\mathrel{\preceq\kern-.5em\raise.22ex\hbox{$\cdot$}}}

\makeatletter
\DeclareFontEncoding{LS1}{}{}
\DeclareFontSubstitution{LS1}{stix}{m}{n}
\DeclareMathAlphabet{\rhomalpha}{LS1}{stixscr}{m}{n}
\makeatother

\nc{\Spa}{\on{{Spa}}}
\nc{\Spd}{\on{{Spd}}}
\nc{\tnb}{\psi_{\rm tame}}
\nc{\oM}{\overline{{M}}}
\nc{\op}{{\on{op}}}
\nc{\ad}{{\on{ad}}}
\nc{\alg}{{\on{alg}}}
\nc{\Ad}{{\on{Ad}}}
\nc{\Adm}{{\on{Adm}}} \nc{\aff}{{\on{af}}}
\nc{\Aut}{{\on{Aut}}}
\nc{\Bun}{{\on{Bun}}}
\nc{\cha}{{\on{char}}}
\nc{\der}{{\on{der}}}
\nc{\Der}{{\on{Der}}}
\nc{\diag}{{\on{diag}}}
\nc{\End}{{\on{End}}}
\nc{\Fl}{{\calF\!\ell}}
\nc{\Tr}{{\on{Transp}}}
\nc{\TR}{{\calT\!\calR}}
\nc{\Gal}{{\on{Gal}}}
\nc{\Gr}{{\on{Gr}}}
\nc{\Hk}{{\on{Hk}}}
\nc{\rH}{{\on{H}}}
\nc{\Hom}{{\on{Hom}}}
\nc{\id}{{\on{id}}}
\nc{\Id}{{\on{Id}}}
\nc{\ind}{{\on{ind}}}
\nc{\Ind}{{\on{Ind}}}
\nc{\Lie}{{\on{Lie}}}
\nc{\Pic}{{\on{Pic}}}
\nc{\pr}{{\on{pr}}}
\nc{\Res}{{\on{Res}}}
\nc{\res}{{\on{res}}} \nc{\Sat}{{\on{Sat}}}
\nc{\spc}{{\on{sc}}}
\nc{\Sp}{{\on{sp}}}
\nc{\drv}{{\on{der}}}
\nc{\sgn}{{\on{sgn}}}
\nc{\Spec}{{\on{Spec}}\,}
\nc{\Spf}{\on{Spf}} 
\nc{\Sph}{\on{Sph}}
\nc{\St}{{\on{St}}}
\nc{\tr}{{\on{tr}}}
\nc{\Mod}{{\mathrm{-Mod}}}
\nc{\Hilb}{{\on{Hilb}}} 
\nc{\Ext}{{\on{Ext}}} 
\nc{\vs}{{\on{Vec}}}
\nc{\ev}{{\on{ev}}}
\nc{\nO}{{\breve{\calO}}}
\nc{\tS}{{\tilde{S}}}
\nc{\spe}{{\on{sp}}}
\nc{\loc}{{\on{loc}}}
\nc{\pre}{{\on{pre}}}

\nc{\dimt}{{\on{dim.trg}}}

\nc{\co}{\colon}
\nc{\dia}{{\diamondsuit}}

\nc{\nscrR}{{\mathscr{R}^{\on{nr}}}}

\nc{\GL}{{\on{GL}}}

\nc{\Gl}{\on{Gl}} 
\nc{\GSp}{{\on{GSp}}}
\nc{\gl}{{\frakg\frakl}}
\nc{\SL}{{\on{SL}}} 
\nc{\SU}{{\on{SU}}} 
\nc{\SO}{{\on{SO}}}
\nc{\PGL}{{\on{PGL}}}

\nc{\Conv}{{\on{Conv}}}
\nc{\Rep}{{\on{Rep}}}
\nc{\Dom}{{\on{Dom}}}
\nc{\red}{{\on{red}}}
\nc{\For}{{\on{for}}}
\nc{\Red}{{\on{Red}}}
\nc{\thi}{{\on{th}}}
\nc{\an}{{\on{an}}}
\nc{\act}{{\on{act}}}
\nc{\nr}{{\on{nr}}}
\nc{\ctf}{{\on{ctf}}}

\nc{\str}{{\on{-}}} 
\nc{\os}{{\bar{s}}}
\nc{\oeta}{{\bar{\eta}}}

\nc{\hookto}{\hookrightarrow}
\nc{\longto}{\longrightarrow}
\nc{\leftto}{\leftarrow}
\nc{\onto}{\twoheadrightarrow}
\nc{\lonto}{\twoheadleftarrow}

\nc{\pot}[1]{ [\hspace{-0,5mm}[ {#1} ]\hspace{-0,5mm}] }
\nc{\rpot}[1]{ (\hspace{-0,7mm}( {#1} )\hspace{-0,7mm}) }
\nc{\smallpot}{{ <\hspace{-1,0mm}<}}

\numberwithin{equation}{section}

\begin{document}
	
\title{Stacks of $p$-adic shtukas and spatial kimberlites.}
	
	\author[I. Gleason]{Ian Gleason}

	\address{National University of Singapore, 10 Lower Kent Ridge Road, Singapore 119076}
	\email{ianandreigf@nus.edu.sg}
	
	\begin{abstract}
		The main purpose of this article is to show that the special Newton polygon map from the stack of $p$-adic shtukas to the stack of $G$-bundles on the Fargues--Fontaine curve is representable in diamonds and sufficiently nice for cohomological considerations (i.e. fdcs). 
		The second purpose is to show that the $\bar{\bbF}_p$-fibers of the special Newton polygon map behave like formal schemes, and in particular, satisfy henselianity properties with respect to their reduced locus. 
		These two goals achieved in this article are two of the crucial ingredients used in our collaboration with Hamman, Ivanov, Louren\c{c}o and Zou to construct the equivalence that compares the schematic and analytic local Langlands categories of Zhu and of Fargues--Scholze.  
		To achieve these goals, we introduce and study spatial kimberlites, which is a better behaved variant of the theory previously developed by the author.
	\end{abstract}

	\maketitle
	\tableofcontents
	
	\section{Introduction}
	Since their introduction by Drinfeld, moduli spaces of shtukas have played a prominent role in our understanding of the Langlands program for characteristic $p$ global fields \cite{DrinfeldGL2,LafforgueGLn,VLafforgueGlobalG}. 
	Shtukas themselves seem to also be intimately related to Grothendieck's vision of motives, and to some extent this is how they were discovered. 
	Indeed, the concept of a Drinfeld module \cite{DrinfeldEllipticModules}, which is a precursor to the concept of a shtuka, stems from Drinfeld's desire to construct an analogue of the modular curve, which in the Langlands program we like to think of as the moduli space of motives attached to elliptic curves. 
	Although we are far from realizing Grothendieck's vision of motives, and we have not learned how to make mathematical sense of classifying spaces of motives, Shimura varieties and stacks of shtukas are avatars of these sought-after classifying spaces, and to date they keep guiding our investigations in the Langlands program. 
	It is through this line of reasoning, by sequences of analogies, and by relying on the deep developments in $p$-adic Hodge theory and perfectoid geometry that Scholze arrived to the definition of the stack of $p$-adic shtukas that we will denote by 
	\[\Sht_\calG\]
	in this article.

	Vaguely speaking, $p$-adic Hodge theory aims to understand the theory of $p$-adic motives, and although the community already recognized and was guided by strong analogies between Drinfeld shtukas and the structures appearing in $p$-adic Hodge theory, it was only after the introduction of perfectoid spaces and Scholze's reinterpretation of the tilting correspondence that one could write down $\Sht_\calG$ in exactly the same footing as Drinfeld's shtukas which is one of the main aims of \cite{SW20}.\footnote{Strictly speaking, in \cite{SW20}, the definition of $\Sht_\calG$ appears only indirectly. Indeed, although the introduction of \S 23 seems to prepare the reader to understand the definition of $\Sht_\calG$ the only definition provided is \cite[Definition 23.1.1]{SW20} for reasons that are unknown to us. To the author's best knowledge, the first time the v-stack $\Sht_\calG$ appeared in the literature treated as a geometric object is in \cite{GI_Bunmer,zhang2023pel} although several many previous references already study the functor of points of this stack.}  

	Although Scholze's definition of $\Sht_\calG$ as the classifying space of $p$-adic motives is far from the final one, and our understanding of $p$-adic Hodge theory keeps improving over the years, the author of this article believes that, for the purposes of studying the $\ell$-adic \'etale cohomology of such a stack (in the case $\ell\neq p$), the version that appears in \cite{SW20} is already the correct geometric object to study. 
\subsection{Why do we study $\calD_{\acute{e}t}(\Sht_\calG)$?}
The study of the \'etale cohomology of local Shimura varieties and their relation to the local Langlands correspondence has a rich tradition.
It started with Drinfeld's work \cite{drinfel1976coverings} which studied the so-called Drinfeld's upper half-space $\Omega_E^d$. 
Drinfeld conjectured that its cohomology realized all supercuspidal representations of $\mathrm{GL}_d(E)$; this conjecture was later made precise by Carayol \cite{carayol1990non} who formulated it in terms of an action by the product group $W_E\times D^\times \times \mathrm{GL}_d(E)$ whose description realized simultaneously the local Langlands and Jacquet--Langlands correspondences.
Kottwitz, following pioneering works of Carayol, Faltings, Genestier, and Harris, formulated a conjecture \cite[Conjecture 5.1]{rapoport1995non} which has become a landmark in the area. 
In their book \cite{RZ96}, Rapoport--Zink construct and study in a systematic way a family of spaces that are now known as Rapoport--Zink spaces. 
One of the main reasons that motivated Rapoport and Zink to undertake this task was to construct the geometric spaces that the Kottwitz' conjecture talked about. 
Rapoport--Zink spaces are moduli spaces of $p$-divisible groups with additional structures, their theory is very rich and their study has been a very active area of research in arithmetic geometry since their introduction to this day.

About a decade ago, the theory went through a very serious transformation following Scholze's introduction of perfectoid spaces \cite{Sch12}, Fargues--Fontaine's reformulation of $p$-adic Hodge theory through the curve \cite{FF18}, and Scholze--Weinstein's perfectoid Dieudonn\'e theory \cite{SW13}.  
These developments motivated Rapoport and Viehmann to write \cite{RV14} where they point towards a first formulation of local Shimura varieties detached from its connection to $p$-divisible groups.
Shortly after, in his Berkeley lectures \cite{SW20}, Scholze introduces $p$-adic shtukas and explains a general construction of local Shimura varieties by reinterpreting them as moduli spaces of $p$-adic shtukas (a.k.a $p$-adic motives). 
The final transformative input came with Fargues' reinterpretation of the local Langlands correspondence as \textit{``geometric Langlands for the Fargues--Fontaine curve''}.
Indeed, in \cite{fargues_geometrization_of_the_local_langlands_correspondence_overview} Fargues explains how the Kottwitz conjecture would follow from the existence of eigensheaves (the geometric Langlands version of eigenforms), objects that geometric Langlands had been studying for traditional curves.     
Fargues' insight materialized in the monumental work of Fargues--Scholze \cite{FS21} in which the categorical local Langlands conjecture (CLLC) is formulated, and which settles the modern foundations of the field. 
In this framework, Fargues and Scholze describe the cohomology of local Shimura varieties attached to the tuple $(G,b,\mu,K)$ as computing the $b$-fiber of the $\mu$-Hecke operator applied to the $K$-compact induction of the trivial representation \cite[\S IX.3]{FS21}.

In summary, from the modern perspective, we study $\calD_{\acute{e}t}(\Sht_\calG)$, or more precisely, the map 
\[\sigma_!:\calD_{\acute{e}t}(\Sht_\calG)\to \calD_{\acute{e}t}(\Bun_G)\]
in order to compute Hecke operators and understand the categorical local Langlands conjecture, something that we had secretly being doing for decades without being completely aware of it.

In this framework, the stack $\Sht_\calG$ (more precisely its bounded version $\Sht_{\calG,\mu}$) ``glues'' all the integral local Shimura varieties into one map, that we call the special Newton polygon map and that we denote by 
\[\sigma:\Sht_\calG\to \Bun_G.\]
Indeed, given integral local Shimura datum $(G,\calG,b,\mu)$ one has (up to isomorphism) a unique map $b:\Spd(\bar{\bbF}_p)\to \Bun_G$ and the $\bar{\bbF}_p$-fiber
\begin{center}
\begin{tikzcd}
	\Sht_{\calG,\mu}(b)	\arrow{r} \arrow{d}  & \Spd(\bar{\bbF}_p) \arrow{d} \\
 \Sht_{\calG,\mu} \arrow{r} &\Bun_G 
\end{tikzcd}
\end{center}
is the integral local Shimura variety attached to $(G,\calG,b,\mu)$.
\\

Let us add another answer to the question: \textit{Why should we try to understand $\calD_{\acute{e}t}(\Sht_\calG)$ and $\sigma_!$?}.

In recent years, the geometric Langlands community have advertised the perspective that the categorical local Langlands conjecture should follow from a $2$-categorical geometric local Langlands conjecture by a formal procedure (taking a trace of Frobenius in a higher categorical sense). 
The tame part of this perspective has been put forward in the recent work of Zhu \cite{zhu2025tamecategoricallocallanglands}.

A priori, the connection between the categorical local Langlands conjecture coming from geometric Langlands and Fargues--Scholze's formulation are unrelated.  
Nevertheless, our main motivations to study
\[\sigma_!:\calD_{\acute{e}t}(\Sht_{\calG,\mu})\to \calD_{\acute{e}t}(\Bun_G)\]
is because this map provides the bridge that connects the two formulations.  
Indeed, in the past few years we have dedicated our work to set foundations to explain precisely how the local geometric Langlands relates the Fargues--Scholze's formulation of CLLC. 
A substantial part of this program is already achieved in \cite{GI_Bunmer}, in our forthcoming collaborative work \cite{GHILZ25}, and in the present article. 
We will continue this line of thought in future works.
\subsection{The main result.}
We fix some notation.
Let $p$ and $\ell$ be prime numbers with $p\neq \ell$.
Let $\Lambda$ be a $\bbZ/\ell^n\bbZ$-algebra.
Let $E$ be a $p$-adic local field with ring of integers $O_E$, uniformizer $\pi\in O_E$, and residue field $\bbF_q=O_E/\pi$. 
Let $\calG$ be a parahoric group scheme over $O_E$ with reductive generic fiber $G$. 
Let $\mu:\bbG_m\to G$ be a conjugacy class of geometric cocharacters with field of definition $F$.
Let $\breve{F}$ denote the completion of the maximal unramified extension of $F$.
Let $\Sht_{\calG,\mu}\to \Spd O_{\breve{F}}$ denote the stack of local shtukas with modification bounded by $\mu$ (see \S \ref{section on stack of shtukas}).
Let $\Bun_G$ denote the stack of $G$-bundles on the Fargues--Fontaine curve \cite[Definition I.2.4]{FS21}.
Let $b\in B(G)$ be an element of the Kottwitz set.
It defines a $G$-bundle $\calE_b\in \Bun_G(\Spd \overline{\bbF}_q)$. 
Consider the special Newton polygon map 
\[\sigma:\Sht_{\calG,\mu}\to \Bun_G.\]
We let $\Sht_{\calG,\mu}(b):=\Sht_{\calG,\mu}\times_{\Bun_G,\calE_b} \ast$.  
\begin{theorem}{(\Cref{shtukasareactualllyspatial})}
	\label{shtukasareactualllyspatial2}
	The following statements hold.
	\begin{enumerate}
		\item The map $\sigma:\Sht_{\calG,\mu}\to \Bun_G$ is representable in locally spatial diamonds (more precisely it is fdcs as in \cite[Definition 5.4]{Mann2022NuclearSheaves}). 
			In particular,
			\[\sigma_!:\calD_{\acute{e}t}(\Sht_{\calG,\mu},\Lambda)\to \calD_{\acute{e}t}(\Bun_G,\Lambda)\]
			exists.
		\item The v-stack $\Sht_{\calG,\mu}$ is an Artin v-stack (see \cite[Definition IV.1.1]{FS21}).  
		\item The v-sheaf $\Sht_{\calG,\mu}(b)$ is a locally spatial kimberlite\footnote{See \S \ref{intro the methods} in the introduction below for an explanation of this term.}  whose reduced special fiber is the affine Deligne--Lusztig variety $X_{\calG,\mu}(b)\simeq \Sht_{\calG,\mu}(b)^\red$. 
	\end{enumerate}
\end{theorem}

\begin{remark}
	Formulating and showing \Cref{shtukasareactualllyspatial2}.(3) was the author's original PhD. thesis goal. A weaker statement was already shown in \cite{Gle24}. 
\end{remark}

\begin{remark}
	\label{relevance LLC}
	That the map $\sigma:\Sht_{\calG,\mu}\to \Bun_G$ is fdcs (\Cref{shtukasareactualllyspatial2}.(1)) plays a decisive role in our collaborative work \cite{GHILZ25} in which we define a comparison functor
	\begin{equation}
\pitch:\on{Shv}^!(\calB(G),\Lambda)\overset{\simeq}{\to} \calD_{\acute{e}t}(\Bun_G,\Lambda).
	\end{equation}
Here the left-hand side is the schematic local Langlands category defined and studied extensively by Zhu \cite{zhu2025tamecategoricallocallanglands}, and the right-hand side is the analytic local Langlands category defined and studied extensively by Fargues and Scholze \cite{FS21}. 
The functor $\pitch$ is defined as an appropriately functorial colimit of functors of the form $\sigma_!$ precomposed with an analytification functor. 
Making precise sense of this colimit is one of the challenges tackled in \cite{GHILZ25}.
\end{remark}

\subsection{The methods.}
\label{intro the methods}
We prove \Cref{shtukasareactualllyspatial2} by a chain of implications 
\[(3)\implies (1) \implies (2),\]
and then we show $(3)$.
The implication $(1)\implies (2)$ follows directly from \cite[Proposition IV.1.8.(iii)]{FS21} and \cite[Theorem IV.1.19]{FS21}. 
The implication $(3)\implies (1)$ follows from \cite[Lemma 13.5]{Sch17} and the fact that maps of locally spatial kimberlites are always representable in locally spatial diamonds (see \Cref{representabilityinlocallyspatialkimberlites}).
Most of our work goes into formulating and showing $(3)$ (i.e. developing a theory of spatial kimberlites). \\

	Kimberlites were introduced by the author in \cite{Gle24} to compute the connected components of local Shimura varieties and moduli spaces of shtukas \cite{Gle21}.
	Kimberlites are analogues of formal schemes in Scholze's theory of diamonds and v-sheaves \cite{Sch17}.

	The theory of kimberlites has had measurable success answering questions in the theory of local and global Shimura varieties \cite{Gle21}, \cite{GL22}, \cite{AGLR22}, \cite{gleason2022connected}, \cite{PR21}, \cite{pappas2024integrallocalshimuravarieties} \cite{takaya2025relativerepresentabilityparahoriclevel}.
	One reason for this, is that many constructions in $p$-adic Hodge theory are often better understood when one restricts the test category to be the category of perfectoid spaces. 
	This naturally gives rise to diamonds, v-sheaves and v-stacks in the site of perfectoid spaces. 
	Often, some of these v-sheaves are natural examples of kimberlites.


	For many purposes, the theory of kimberlites already developed in \cite{Gle24}, \cite{AGLR22}, \cite{GL22} and \cite{GI_Bunmer} provides sufficient abstract tools. 
	This is the case specially when one only needs to work with geometric objects that are ``$\pi$-adic''. 

Unfortunately, the theory of kimberlites falls short in answering some foundational questions outside the $\pi$-adic case (see \Cref{naturalquestions}).
From the point of view of the theory of local Shimura varieties, Rapoport--Zink spaces and integral models of moduli spaces of local shtukas, this is problematic since these objects are only rarely $\pi$-adic.

Spatial kimberlites (\Cref{spatial-kimberlite-defi}) fix many of the shortcomings of the theory of kimberlites.
For example, we can prove that formal completions of spatial kimberlites are again spatial kimberlites (see \Cref{keytubularneighborhood}), and that the category of spatial kimberlites is stable under finite limits (see \Cref{stable-under-cartesian}). 
And as we have alluded already, every map of spatial kimberlites is automatically representable in locally spatial diamonds (see \Cref{representabilityinlocallyspatialkimberlites}). This also shows that spatial kimberlites are automatically Artin v-sheaves.
That the category of spatial kimberlites is stable under such natural constructions plays a decisive role in showing that each $\Sht_{\calG,\mu}(b)$ is a locally spatial kimberlite. 
Indeed, we unfold the geometry of $\Sht_{\calG,\mu}(b)$ using these types of natural constructions, and it is important to know that in each of the steps the property of being a (locally) spatial kimberlite is preserved.
To unfold the geometry of $\Sht_{\calG,\mu}(b)$ we rely on results from \cite{FS21}, \cite{AGLR22}, \cite{Gle21}, \cite{güthge2024perfectprismaticfcrystalspadicshtukas}, and \cite{GI_Bunmer}.

\subsection{Spatial kimberlites and cohomology.}
Although the main goal of the article is to show the existence of the map 
\[\sigma_!:\calD_{\acute{e}t}(\Sht_{\calG,\mu},\Lambda)\to \calD_{\acute{e}t}(\Bun_G,\Lambda),\]
we also show a theorem which gives sufficient control on this operation, which is our second main theorem.
Intuitively speaking, it shows that spatial kimberlites, in similarity with formal schemes, are henselian.
Let us clarify what we mean by this.

Let $X$ be a spatial kimberlite and suppose that the reduced locus $X^\red$ (see \cite[Definition 3.12]{Gle24}) is representable by the perfection of a scheme that is proper and perfectly finitely presented over $\Spec \overline{\bbF}_p$. 
We let $\ast=\Spec \overline{\bbF}_p$.
Let $\pi:X\to \ast$ denote the structure map and let $j:X^{\on{an}}\to X$ denote the inclusion of the analytic locus.
Suppose that $X^{\on{an}}\to \ast$ is of $\on{tr.deg.}<\infty$.
Our result roughly says that $X$ as above satisfies the two henselian properties along $X^\red$.
More precisely the result reads as follows.
\begin{theorem}{(\Cref{nearbycyclestheorem})}
	\label{nearbycyclestheorem2}
	Let the notation be as above.
	For all $A\in \calD_{\acute{e}t}(X^{\on{an}},\Lambda)$ we have that 
	\[\pi_!j_*A\simeq 0\simeq \pi_*j_!A.\]
\end{theorem}

\begin{remark}
	\Cref{nearbycyclestheorem2} combined with \Cref{shtukasareactualllyspatial2}.(3) is the main cohomological tool to study the fibers of $\sigma_!$, and provides one of the main computational techniques used in \cite{GHILZ25} to show that the functor $\pitch$ is an equivalence.   
\end{remark}

	\subsection*{Acknowledgements}
	This paper was written during stays at Max-Planck-Institut für Mathematik, Universität Bonn and the National University of Singapore, we are thankful for the hospitality of these institutions. 
	The project has received funding by DFG via the Leibniz-Preis of Peter Scholze, and from NUS PYP funding scheme.
	We would like to thank Peter Scholze and Pol van Hoften for conversations related to this work. 
	
	\section{Notation}
	We start by fixing the context. 
	    For the rest of the paper we fix a prime $p\in \bbZ$.
	    We fix a perfect field $k$ in characteristic $p$, sometimes we will explicitly restrict to the cases $\bbF_p$ or $\overline{\bbF}_p$.
	    We work in the category of perfectoid spaces over $k$ \cite[$\mathsection 3$]{Sch17} endowed with the v-topology \cite[Definition 8.1]{Sch17}, which we denote by $\Perf$.
	    We let $\PSch$ denote the category of perfect schemes over $k$ endowed with the schematic v-topology \cite{bhatt_scholze_projectivity_of_the_witt_vector_affine_grassmannian}.
	    All of our geometric objects will be either small v-stacks on $\Perf$ \cite[$\mathsection 12$]{Sch17} or small scheme theoretic v-stacks on $\PSch$.
	    We denote these categories by $\widetilde{\Perf}$ and $\widetilde{\PSch}$ respectively.
	    The final objects are $\Spd(k,k)$ and $\Spec k$ respectively and all constructions are performed with respect to this base. 
	    We use the symbol $\ast$ to denote these final objects. 
	    From now on we omit $k$ from the notation.

	    Unless we explicitly say otherwise, we restrict the notation $\Spa(R,R^+)$ to denote an affinoid perfectoid space. 
	    For more general Huber pairs $(A,A^+)$, we will mostly only consider its associated v-sheaf $\Spd(A,A^+)$ (see \cite[Definition 10.1.1]{SW20}).
	    When $(A,A^+)$ is a Huber pair such that $A^+=A^\circ$ we write $\Spd A$ instead of $\Spd(A,A^+)$ to denote the v-sheaf attached to $(A,A^+)$.

   Let $B$ be a perfect characteristic $p$ topological ring endowed with the $I$-adic topology for a finitely generated ideal $I\subseteq B$.
   We call any v-sheaf of the form $\Spd(B,B)$ an \textit{affine formal v-sheaf}. 
   Since $B=B^\circ$, following our convention above, we will write $\Spd B$ instead of $\Spd(B,B)$. 
   An important example, is when $(R,R^+)$ is a perfectoid Huber pair. 
   In this case, $\Spd R^+$ is an affine formal v-sheaf associated to the topological ring $R^+$.

   Given a perfectoid Huber pair $(R,R^+)$ we write $R^+_\red$ for $(R^+/R^{\circ \circ})$, which we will consider as a ring endowed with the discrete topology. 
   One can verify that this is a perfect $k$-algebra. 

	\section{Theory of kimberlites}
	The purpose of this paper is to introduce and study spatial kimberlites. 
	This notion refines that of prekimberlites and of kimberlites introduced in \cite{Gle24}. 
	In this section, we review some of the basic theory of kimberlites. 
	For the sake of completeness, we also write down some statements that although easy to deduce are not explicitly stated in \cite{Gle24}.
	We review prekimberlites, valuative prekimberlites, kimberlites and we finish by discussing spatial kimberlites.
	\subsection{Prekimberlites}
	Recall that the rule $\Spec A\mapsto \Spd(A,A)$ extends to a fully faithful functor $\diamond:\PSch\to \widetilde{\Perf}$ \cite[Proposition 18.3.1]{SW20} \cite[$\mathsection 3$]{Gle24}, and that this functor extends uniquely to a continuous functor $\diamond:\widetilde{\PSch}\to \widetilde{\Perf}$. 
Moreover, $\diamond$ admits a right adjoint functor \cite[Definition 3.12]{Gle24} 
	\[\red:\widetilde{\Perf}\to \widetilde{\PSch}.\]
	We let $X^\Red=(X^\red)^\diamond$, it comes with a canonical map coming from adjunction $X^\Red\to X$.
	\begin{definition}{\cite[Definition 3.20]{Gle24}}
		A map $X\to Y$ is \textit{formally adic} if the following diagram is Cartesian 
		\begin{center}
		\begin{tikzcd}
		 Y^\Red \arrow{r} \arrow{d}  & \arrow{d} Y \\
		 X^\Red \arrow{r} & X.
		\end{tikzcd}
		\end{center}
	\end{definition}

	\begin{definition}
		If $f:X\to Y$ is a closed immersion of v-sheaves, we say that $X$ is \textit{formally closed} if $f$ is formally adic.
	\end{definition}

\begin{definition}{\cite[Definition 3.27]{Gle24}}
	A v-sheaf is \textit{formally separated} if the diagonal is a closed immersion and formally adic (i.e. a formally closed immersion).	
	\end{definition}

	Here is a key lemma used throughout \cite{Gle24} that unfortunately has an awkward formulation \textit{loc.cit.} (see \cite[Lemma 3.31]{Gle24}).
	The slogan is that $\Spa(R,R^+)$ is formally dense in $\Spd(R^+)$. 

	\begin{lemma}
		\label{fixed up version formally dense}
		If $\Spa(R,R^+)$ is a perfectoid space and $\calF\subseteq \Spd(R^+)$ is a formally closed subsheaf containing $\Spa(R,R^+)$, then $\calF=\Spd(R^+)$.
	\end{lemma}
	\begin{proof}
		Since the map is a closed immersion, it is qcqs and it suffices to show that it is surjective on geometric points (see \cite[Lemma 12.11]{Sch17}). 	
		By hypothesis, $\calF$ contains the open subset $\Spa(R,R^+)\subseteq \Spd(R^+)$, so it suffices to show that $\calF^\Red\to \Spd(R^+_\red)$, but this is the content of \cite[Lemma 3.31]{Gle24}. 
	\end{proof}

	\begin{definition}{\cite[Definition 4.6]{Gle24}}
		Given a v-sheaf $X$ and a map $f:\Spa (R,R^+)\to X$ we say that $X$ \textit{formalizes} $f$ if there exists a dashed arrow completing the commutative diagram below
		\begin{center}
		\begin{tikzcd}
			\Spa (R,R^+)\arrow{r}{f} \arrow{d}  & X  \\
		 \Spd R^+\arrow[dashed]{ru}  .
		\end{tikzcd}
		\end{center}
		Any such arrow is called a \emph{formalization} of $f$.
		We say that $X$ is \textit{v-formalizing} if for any $f$ as above there is a v-cover $g:\Spa( R',R'^+)\to \Spa (R,R^+)$ such that $X$ formalizes $g\circ f$. 
	\end{definition}

	Whenever $X$ is formally separated and $f:\Spa(R,R^+)\to X$ is a map, a formalization of $f$, if it exists, is unique \cite[Proposition 4.9]{Gle24}.

	\begin{definition}{\cite[Definition 4.11]{Gle24}}
		\label{specializing v-sheaf}
		We say that $X$ is \textit{specializing} if it is formally separated and v-formalizing.	
	\end{definition}

As shown in \cite[Proposition 4.14]{Gle24} if $X$ is a specializing v-sheaf then it has a continuous specialization map 
	$\on{sp}:|X|\to |X^\red|$, but not much more can be said about general specializing v-sheaves.

	\begin{definition}{\cite[Definition 4.15]{Gle24}}
		\label{defiitioon f prekimbelrite}
		If $X$ is a specializing v-sheaf we say that it is a \textit{prekimberlite} if $X^\red\in \widetilde{\PSch}$ is representable by a perfect scheme and $X^\Red\to X$ is a closed immersion.
		We let $X^{\on{an}}:=X\setminus X^\Red$ and we call this the \textit{analytic locus} of $X$.
	\end{definition}

    \begin{definition}
	    Let $X$ be a prekimberlite.
	    \begin{enumerate}
    \item We say that $X$ is an \textit{affine prekimberlite} if $X^\red$ is an affine scheme.
    \item We say that $X$ is a \textit{pointed prekimberlite} if $|X^\red|$ has one element.
	    \end{enumerate}
    \end{definition}

    \begin{proposition}
	    \label{stabilityprekimberlites}
	    The categories of prekimberlites and affine prekimberlites are stable under fiber product and contain $\ast$.	
	    In particular, these categories are stable under finite limits.
    \end{proposition}
    \begin{proof}
	    Let $X$, $Y$ and $Z$ be prekimberlites, and let $W=X\times_Y Z$.
	    By \cite[Proposition 4.10]{Gle24}, the category of specializing v-sheaves is stable under fiber products.
	    Moreover, since the reduction functor commutes with limits, it is clear that $W^\red$ is a scheme (respectively affine scheme) as long as $X^\red$, $Y^\red$ and $Z^\red$ are. 
	    We write $W$ as the base change of the map $X\times Z\to Y\times Y$ along the diagonal of $Y$.
	    Since $Y$ is formally separated, the map $W\to X\times Z$ is formally adic and it suffices to show that $X^\Red \times Z^\Red\to X\times Z$ is a closed immersion, but this is clear since $X^\Red\to X$ and $Z^\Red\to Z$ are. 
    \end{proof}

	For prekimberlites one can construct a v-sheaf theoretic specialization map (or Heuer specialization map). 
Given $S\in \PSch$ one can construct a v-sheaf $S^{\diamond/\circ\circ}$ \cite[Definition 4.23]{Gle24}\footnote{The symbol $\diamond/\circ\circ$ suggest the similarity with the $\diamond$ functor up to a quotient by the space of topological nilpotent elements. In \cite{Gle24} we initially took the minimalistic notation $\diamond/\circ$ instead of $\diamond/\circ\circ$.}, \cite[Definition 5.1]{Heu21} by v-sheafifying the formula
\[S^{\diamond/\circ\circ_{\on{pre}}}:\Spa (R,R^+)\mapsto S(\Spec R^+_\red).\]
As it turns out it is only necessary to sheafify for the analytic topology \cite[Lemma 5.2]{Heu21}. 
When $X$ is a prekimberlite, we let $X^{\on{H}}:=(X^\red)^{\diamond/\circ\circ}$. 
We have a v-sheaf theoretic specialization map \cite[$\mathsection 4.4$]{Gle24} 
\[\on{SP}_X:X\to X^{\on{H}},\]
that we define as follows.
\begin{definition}
	\label{construction H specialization}
The v-sheaf theoretic specialization map is constructed as follows. If $\alpha\in X(R,R^+)$ and $\alpha$ is formalizable we let $\tilde{\alpha}\in X(\Spd R^+)$ be its unique formalization.
Applying the reduction functor gives $\tilde{\alpha}^\red\in X^\red(\Spec R^+_\red)$, which is an element in the presheaf $(X^{\red})^{(\diamond/\circ\circ)_{\on{pre}}}(R,R^+)$.
This describes a map $\on{SP}_{X,\on{pre}}:X^{\on{frml}}\to (X^\red)^{(\diamond/\circ\circ)_{\on{pre}}}$, where the source is the sub-presheaf of formalizable maps in $X$.
Now, $\on{SP}_X$ is the sheafification of $\on{SP}_{X,\on{pre}}$.
Unless we need to clarify which respect to which space we take the Heuer specialization map, we will omit $X$ from the notation and simply write $\on{SP}$.
\end{definition}

\begin{remark}
	The reader may notice that the construction of \Cref{construction H specialization} can be applied to any specializing v-sheaf. 
	But the remarks concerning sheafification with respect to the analytic topology (i.e. \cite[Lemma 5.2]{Heu21}) no longer hold in this generality.
\end{remark}

We have the following lemma which although does not mention \Cref{construction H specialization} directly, it secretly uses it. 
It essentially says that for specializing sheaves, being a prekimberlite is a Zariski local condition.

\begin{lemma}{\cite[Lemma 2.32]{Gle21}}
	\label{lemmaspecializing}
	Let $X$ be a specializing v-sheaf and let $Y=X^\red$. 
	Suppose that $Y$ is representable by a perfect scheme. 
	Let $U\subseteq Y$ be an open subset and let $V$ denote the only open subsheaf of $X$ with $|V|=\on{sp}^{-1}(|U|)$. 
	Then the following hold.
	\begin{enumerate}
		\item $V$ is a specializing v-sheaf with $V^\red=U$.
		\item The map $V\to X$ is formally adic.
	\end{enumerate}
	Moreover, if there is an open cover $\{U_i\to Y\}_{i\in I}$ such that $V_i:=\on{sp}_X^{-1}(U_i)$ is a prekimberlite. Then $X$ is a prekimberlite. 
\end{lemma}

If $X$ is a prekimberlite and $Z\subseteq X^\red$ is a locally closed subset we can define a prekimberlite $\widehat{X}_{/Z}$ \cite[Proposition 4.21]{Gle24} such that $Z=(\widehat{X}_{/Z})^\red$ and such that it fits in the following Cartesian diagram 
\begin{equation}
	\label{formalnbhoods}
\begin{tikzcd}
	\widehat{X}_{/Z} \arrow{r} \arrow{d}{\on{SP}}  & X \arrow{d}{\on{SP}} \\
 Z^{\diamond/\circ\circ} \arrow{r} & X^{\on{H}}. 
\end{tikzcd}
\end{equation}
When $Z$ is constructible, the map $\widehat{X}_{/Z}\to X$ is an open immersion \cite[Proposition 4.22]{Gle24}, but it is often \textbf{not} formally adic.
Intuitively speaking, the rule $X\mapsto \widehat{X}_{/Z}$ corresponds to taking the completion of $X$ along $Z$.
\begin{definition}{\cite[Definition 4.18, 4.38]{Gle24}}
	\label{formalnbbhoosoughtotbedefined}
	The prekimberlite $\widehat{X}_{/Z}$ obtained from diagram \eqref{formalnbhoods} is called the \textit{formal neighborhood} of $X$ along $Z$. The v-sheaf ${X}^\circledcirc_{/Z}:=X^{\on{an}}\cap \widehat{X}_{/Z}$ is called the \textit{tubular neighborhood} of $X$ along $Z$.
\end{definition}

If $X$ is a prekimberlite and $U\to X^\red$ is an \'etale map of separated schemes we can construct another prekimberlite $\widehat{X}_{/U}$ such that $U=(\widehat{X}_{/U})^\red$ and such that it fits in the following Cartesian diagram
\begin{equation}
	\label{formaletlaenbhosds}
\begin{tikzcd}
	\widehat{X}_{/U} \arrow{r} \arrow{d}{\on{SP}}  & X \arrow{d}{\on{SP}} \\
 U^{\diamond/\circ\circ} \arrow{r} & X^{\on{H}}. 
\end{tikzcd}
\end{equation}
The map is $\widehat{X}_{/U}\to X$ is always formally adic and \'etale. 
Moreover, \cite[Theorem 4.27]{Gle24} intrinsically characterizes all v-sheaves constructed in this way. 
Indeed, if $Y$ is a prekimberlite endowed with a formally adic and \'etale map to $X$, then $Y^\red\to X^\red$ is an \'etale separated map of schemes and \cite[Theorem 4.27]{Gle24} shows that 
\[Y\simeq \widehat{X}_{/Y^\red}.\]
One can regard \cite[Theorem 4.27]{Gle24} as an analogue of the invariance of the \'etale site under nil-thickenings.

\begin{definition}{\cite[Definition 4.26]{Gle24}}
The prekimberlite $\widehat{X}_{/U}$ obtained from diagram \eqref{formaletlaenbhosds} is called the \textit{\'etale formal neighborhood} of $X$ along $U$.
\end{definition}

\begin{remark}
The category of \'etale formal neighborhoods of \cite[Definition 4.26]{Gle24} can be used to define the \textit{naive nearby cycles functor} \cite[Remark 4.29]{Gle24} studied and successfully applied in \cite{GL22}.
\end{remark}

The following statement turned out to be trickier than the author originally expected. 
Luckily for us, Kim's theorem simplifies the situation (see \cite[Theorem 1.3]{kim2024descendingfiniteprojectivemodules}). 

\begin{proposition}
	\label{finite-etale-lemma}
	If $X$ is a prekimberlite and $f:Y\to X$ is a finite \'etale map of v-sheaves, then the following statements hold.
	\begin{enumerate}
		\item $Y$ is a prekimberlite.
		\item $f$ is formally adic.
		\item $Y=\widehat{X}_{/Y^\red}$ (i.e. $Y$ is an \'etale formal neighborhood of $X$).
	\end{enumerate}
\end{proposition}
\begin{proof}
	Let $Z=Y\times_X X^\Red$. 
	Since by hypothesis $X^\Red$ is representable by a scheme, it follows from \cite[Theorem 1.3]{kim2024descendingfiniteprojectivemodules} that $Z$ is also representable by a scheme. 
	More precisely, $Z=T^\diamond$ for $T\to X$ a finite \'etale morphism.
	From \cite[Lemma 3.32]{Gle24}, it follows that $f:Y\to X$ is formally adic, and that $T=Y^\red$ (this already takes care of \Cref{finite-etale-lemma}.(2)). 
	Since $Y\to X$ is separated and $X\to \ast$ is separated, it follows that $\Delta_Y:Y\to Y\times Y$ is a closed immersion.
	By \cite[Lemma 3.30]{Gle24}, $\Delta_Y$ is formally adic. 
	We conclude that $Y$ is formally separated. 
	Since the adjunction map $Y^\Red\to Y$ is a closed immersion, and $Y^\red$ is a scheme, to show that $Y$ is a prekimberlite it suffices to show that $Y$ is v-formalizing.
	Once we know that $Y$ is a prekimberlite, it follows from \cite[Theorem 4.27]{Gle24} that necessarily $Y\simeq \widehat{X}_{/Y^\red}$.

	Let us show that $Y$ is v-formalizing. 
	We first make some reductions. 
	Let $g:\Spa(R,R^+)\to Y$ be a map, without loss of generality we may assume that $X$ formalizes $f\circ g$, and that $\Spa(R,R^+)$ is a product of points since these form a basis for the v-topology (see \cite[Remark 1.3]{Gle24}). 
	It suffices to show that any map $\Spa(R,R^+)\to Y\times_X \Spd(R^+,R^+)$ is formalizable, or in other words, without loss of generality we may and do assume $X=\Spd(R^+,R^+)$.
	Since the degree function of the finite \'etale map $f$ is locally constant, after decomposing $X$ as a union of open and closed subsets, we may reduce to the case that $f$ is of constant degree $n$ for some $n\in \bbN$.
	We consider $Y_1$ to be the v-sheaf of isomorphism between $Y$ and the constant sheaf of cardinality $n$ as v-sheaves over $X$.
	Then $Y_1\to X$ is a finite \'etale $S_n$-torsor.
	Any section $X\to Y_1$ provides an isomorphism $Y\simeq X \times \{1,\dots, n\}$, and one can explicitly construct formalizations for $Y$ of such form.
	In other words, we have reduced the problem to showing that any $S_n$-torsor over $X=\Spd(R^+,R^+)$ is trivial. 
	
	Since the map $Y_1\to X$ is also finite \'etale, it suffices to show that any finite \'etale map towards $X$ admits a section. 
	To ease the notation, we let $Y=Y_1$.
	We reduce to the case $X=\Spd(C^+,C^+)$ with $C^+$ a valuation ring with algebraically closed fraction field as follows.
	Consider $\Spd(R^+\pot{t},R^+\pot{t})$ and consider its analytic locus $\widetilde{X}:=\Spd(R^+\pot{t},R^+\pot{t})\setminus V(t,\varpi)$ where $\varpi\in R^+$ is a pseudo-uniformizer. 
	It is not hard to see that $\widetilde{X}$ is a qcqs perfectoid space and that the map $\widetilde{X}^{\on{an}}\to X$ is a cohomologically smooth surjection (see \cite[\S 23]{Sch17}) with geometrically connected fibers.
	Fix $x\in \pi_0(X)$, for any $U\subseteq \pi_0(X)$ that is open and closed and contains $x$ we have spaces $X_U=\Spd(R^+_U,R^+_U)$, $Y_U:={X_U}\times_X Y$, $\widetilde{X_U}:=X_U\times_X \widetilde{X}$ and $\widetilde{Y_U}:=\widetilde{X_U}\times_X Y$, obtained by cutting along the appropriate idempotent.
	Passing to limits we also have spaces $X_x=\Spd(C_x^+,C_x^+)$, $Y_x$, $\tilde{X}_x$ and $\tilde{Y}_x$.
	Here $C_x^+$ is a valuation ring with algebraically closed non-Archimedean fraction field.
	Suppose we know that $Y_x\to X_x$ admits a section. Then $\tilde{Y}_x\to \tilde{X}_x$ admits a section, and by \cite[Lemma 12.17]{Sch17} this section spreads to a section $\tilde{X_U}\to \tilde{Y_U}$ for some $U$.
	Since the map $\tilde{X_U}\to \tilde{Y_U}$ is finite \'etale, its image is an open and closed subset $\tilde{T}\subseteq \tilde{Y_U}$.
	Moreover, if we let $h$ denote the map $h:\tilde{Y_U}\to Y_U$, then $T=h^{-1}(h(T))$.
	Indeed, the fibers of $h$ are connected and $T$ is already open and closed in $\tilde{Y_U}$.
	We may descend $\tilde{T}$ to an open and closed subset $T\subseteq Y_U$ whose basechange along $h$ is $\tilde{T}$. 
	Then $T\to X_U$ is an isomorphism, since it becomes one after basechange. 
	We have shown that as long as $Y_x\to X_x$ admits a section, this section spreads to a section of $Y_U\to X_U$ for some open and closed subset $U\subseteq \pi_0(X)$, which finishes our reduction step.

	Finally, we show that any finite \'etale map $f:Y\to X$ splits whenever $X=\Spd(C^+,C^+)$ with $C^+$ a valuation ring with algebraically closed fraction field.
	We show that if $Y\to X$ is of degree $n$, then $Y$ has $n$ connected components. 
	Passing to one such component $Y_0$ gives a map $Y_0\to X$ which is finite \'etale of degree $1$, so it must be an isomorphism by \cite[Lemma 12.5]{Sch17}.
	Now, we may compute $\on{H}^0(Y,\bbF_\ell)$ as $H^0(X,f_*\bbF_\ell)$.
	Moreover, we have an exact triangle 
	\begin{equation}
		\label{some-exact-triangle}
	j_!f^{\on{an}}_*\bbF_\ell\to f_*\bbF_\ell \to i_*f^{\on{Red}}_* \bbF_\ell
	\end{equation}
	coming from excision. 
	Here $i:X^{\on{Red}}\to X$ is the closed immersion coming from adjunction, $j:X^{\on{an}}\to X$ is the open complement of $i$ and the maps $f^{\on{an}}:Y^{\on{an}}\to X^{\on{an}}$ and $f^{\on{Red}}:Y^\Red\to X^\Red$ are induced from $f:Y\to X$ by basechange and the fact that (as we saw above) $f$ is formally adic. 

	By \cite[Lemma 4.3]{GL22}, $R\Gamma(X,j_!f^{\on{an}}_*\bbF_\ell)$ vanishes.
	This shows that $Y$ has as many connected components as $Y^\Red$.
	Using Kim's theorem again (\cite[Theorem 1.3]{kim2024descendingfiniteprojectivemodules}), we see that $Y^\Red\simeq X^\Red\times \{1,\dots, n\}$ whenever $f$ has degree $n$.
	Since $X^\Red=\Spec(C^+_\red)^\diamond$, and $C^+_\red$ is a valuation ring with algebraically closed fraction field hence strictly henselian.
\end{proof}

\begin{remark}
One could have avoided using \cite[Theorem 1.3]{kim2024descendingfiniteprojectivemodules} at the expense of making the argument in \Cref{finite-etale-lemma} substantially longer.  
This type of argument boils down to analyzing finite \'etale maps over v-sheaves of the form $\Spd(k^+,k^+)$ where $k^+$ is a valuation ring endowed with the discrete topology. 
This analysis can be approached by similar methods using exact triangles as in \Cref{some-exact-triangle}, and approximation arguments. 
Before \cite[Theorem 1.3]{kim2024descendingfiniteprojectivemodules} was available, we had such an approach in mind. 
Nevertheless, we consider \cite[Theorem 1.3]{kim2024descendingfiniteprojectivemodules} and its proof method a better way to proceed. 
\end{remark}

 \subsection{Valuative prekimberlites}
 Let us recall the following definition \cite[Definition 18.4]{Sch17}.
\begin{definition}
	We say that a map $X\to Y$ of v-sheaves \emph{satisfies the valuative criterion of partial properness} if for any commutative diagram of the form
	\begin{center}
	\begin{tikzcd}
		\Spa(R,R^\circ) \arrow{r} \arrow{d}  & X \arrow{d} \\
		\Spa (R,R^+)\arrow{r} & Y
	\end{tikzcd}
	\end{center}
	with $\Spa(R,R^+)$ affinoid perfectoid, there exists a unique map $\Spa(R,R^+)\to X$ making the diagram above commutative.
	We say that $X\to Y$ is partially proper if it is separated and satisfies the valuative criterion of partial properness.
\end{definition}

The specialization map for prekimberlites, $\on{SP}:X\to X^{\on{H}}$, is always separated. 
Indeed, the inclusion $X\times_{X^{\on{H}}}X\subseteq X\times X$ is a separated map and since $X$ is formally separated $\Delta:X\to X\times X$ is a closed immersion. 
Consequently, $X\to X\times_{X^{\on{H}}}X$ also is.
\begin{definition}{\cite[Definition 4.30]{Gle24}}
	\label{valuative prekimberlite defi}
	We say that a prekimberlite $X$ is \textit{valuative} if $\on{SP}:X\to X^{\on{H}}$ is partially proper.
\end{definition}

For valuative prekimberlites the topological specialization map, 
\[\on{sp}:|X|\to |X^\red|,\] 
is specializing \cite[Proposition 4.33]{Gle24}.
Moreover, the valuative property is stable under natural constructions like taking formal neighborhoods or \'etale formal neighborhoods \cite[Proposition 4.34]{Gle24}. 
We prove some further stability properties of the valuative property.

\begin{proposition}
	\label{formallyadicclosedisvaluative}
	Suppose $X$ is a valuative prekimberlite and that $Z\to X$ is a formally adic closed immersion.
	Then $Z$ is a valuative prekimberlite.
\end{proposition}
\begin{proof}
	By \cite[Proposition 4.41.(2)]{Gle24} $Z$ is a prekimberlite. 
	The map $Z^{\on{H}}\to X^{\on{H}}$ is separated since it is injective.
	The map $Z\to  X^{\on{H}}$ is partially proper and by cancellation $Z\to Z^{\on{H}}$ is partially proper.
\end{proof}

\begin{proposition}
	\label{valuativeprekimbfiberprod}
	The category of valuative prekimberlites is stable under fiber product.	
\end{proposition}
\begin{proof}
If $X=Z\times_W Y$ we can rewrite this as the base change of $Z\times Y\to W\times W$ along $W\to W\times W$ which is a formally adic closed immersion.
By \Cref{formallyadicclosedisvaluative}, it suffices to show that $Z\times Y$ is valuative. 
Now the map $Z\times Y\to Z^{\on{H}}\times Y$ is partially proper, and the map $Z^{\on{H}}\times Y\to Z^{\on{H}}\times Y^{\on{H}}$ also is, so their composition is partially proper as well.
\end{proof}

A technical subtlety when studying the valuative property of prekimberlites is that, in most cases, if $S$ is a separated scheme, then $S^{\diamond/\circ\circ}$ is far from being a separated v-sheaf.
    Nevertheless, they satisfy a weaker form of separatedness. 
    \begin{definition}
	    Let $X\to Y$ be a map of small v-sheaves.
	    \begin{enumerate}
		    \item We say that $X\to Y$ is \textit{weakly separated} if the diagonal $\Delta:X\to X\times_Y X$ is partially proper.
			    When $Y=\ast$ we simply say that $X$ is a \textit{weakly separated v-sheaf}.
		    \item We say that $X\to Y$ is \textit{weakly partially proper} if it is weakly separated and satisfies the valuative criterion of partial properness.  
			    When $Y=\ast$ we simply say that $X$ is a \textit{weakly partially proper v-sheaf}.
	    \end{enumerate}
    \end{definition}
    \begin{remark}
    Note that a map $X\to Y$ is separated (respectively partially proper) if and only if it is weakly separated and quasiseparated (respectively weakly partially proper and quasiseparated).
    \end{remark}

\begin{proposition}
	\label{partial proper v-local}
	Let $f:X\to Y$ be a map of v-sheaves and $g:Z\to Y$ a v-cover with base change map $\tilde{f}:X\times_Y Z \to Z$. 
	If $\tilde{f}$ satisfies the valuative criterion of partial properness, then $f$ satisfies the valuative criterion of partial properness.
	In particular, if $\tilde{f}$ is partially proper, then $f$ is partially proper.
\end{proposition}
\begin{proof}
	That separatedness is v-local is the content of \cite[Proposition 10.11.(ii)]{Sch17}. 
	So we are immediately reduced to showing the first statement.

	Suppose we are given a commutative diagram
	\begin{center}
	\begin{tikzcd}
		\Spa(R,R^\circ) \arrow{r} \arrow{d}  & X \arrow{d} \\
		\Spa (R,R^+)\arrow{r} & Y.
	\end{tikzcd}
	\end{center}
	By assumption, there exists a v-cover $\Spa(R_0,R_0^+)\to \Spa(R,R^+)$ for which we can form a commutative diagram of the form
	\begin{center}
	\begin{tikzcd}
		\Spa(R_0,R_0^\circ) \arrow{r} \arrow{d}  & X\times_Y Z \arrow{d} \ar{r} & X \ar{d}  \\
		\Spa (R_0,R_0^+)\arrow{r} & Z \ar{r} & Y.
	\end{tikzcd}
	\end{center}
	By hypothesis, we can find a lift $\Spa(R_0,R_0^+)\to X\times_Y Z$ and hence we find a map $\Spa(R_0,R_0^+)\to X$.
	Moreover, by assumption this map is unique.
	Uniqueness allows us to descend this to a map of the form $\Spa(R,R^+)\to X$ fitting in the original commutative diagram.
\end{proof}

    \begin{proposition}
	    \label{closedimmersiongivespartiallyproper}
	    Suppose that $Z\to X$ is a closed immersion of schemes, then $Z^{\diamond/\circ\circ}\to X^{\diamond/\circ\circ}$ is a partially proper injection. 
	    In particular, if $X\to Y$ is a separated map of schemes, then $X^{\diamond/\circ\circ}\to Y^{\diamond/\circ\circ}$ is weakly separated. 
    \end{proposition}
    \begin{proof}
	    \Cref{partial proper v-local} says that being partially proper is v-local on the target. 
	    If $U\to X$ is an open cover of schemes, then $U^{\diamond/\circ\circ}\to X^{\diamond/\circ\circ}$ is also an open cover, so we may assume that $X=\Spec A$ and $Z=\Spec B$.
	    A map $\Spa(R,R^\circ)\to Z^{\diamond/\circ\circ}$ is the same as a map of rings $B\to R^\circ_{\red}$, and a map $\Spa(R,R^+)\to X^{\diamond/\circ\circ}$ is the same as a map of rings $A\to R^+_{\red}$. But we have an inclusion of rings $R_\red^\circ\supseteq R_\red^+$, and a surjection of rings $A\to B$. 
	    This constructs the unique lift $B\to R_\red^+$. 
    \end{proof}

    \begin{example}
	    When $X=\bbA^1$, then $X^{\diamond/\circ\circ}$ is not separated over $\ast$. Indeed, let $Y$ fit in the following Cartesian diagram
	    \begin{center}
	    \begin{tikzcd}
	     Y\arrow{r} \arrow{d}  & \arrow{d} X\times X \\
	     X^{\diamond/\circ\circ}\arrow{r} & X^{\diamond/\circ\circ}\times X^{\diamond/\circ\circ}.
	    \end{tikzcd}
	    \end{center}
	    Then $Y$ is the formal neighborhood of $X\times X$ along the diagonal. 
	    This is a partially proper open subsheaf of $X\times X$, but not a closed subsheaf. 
    \end{example}

\begin{lemma}
	\label{enhancedporperness}
Let $A\subseteq K$ be rings with the property that every local ring of $\Spec A$ is a valuation ring and has non-empty intersection with $\Spec K$, and suppose that the map $\Spec K\to \Spec A$ is pro-open. Let $X\to Y$ be a proper map of schemes, then the following commutative diagram has a unique solution for the dashed arrow
\begin{center}
\begin{tikzcd}
	\Spec K\arrow{r}{f} \arrow{d}  &X \arrow{d} \\
 \Spec A\arrow{r}\ar[dashed]{ur} &  Y.
\end{tikzcd}
\end{center}
\end{lemma}
\begin{proof}
Uniqueness follows from separatedness since by hypothesis $A\subseteq K$, and $\Spec K$ is Zariski dense in $\Spec A$. 
Without losing generality we may assume $Y=\Spec A$. 
Let $Z\subseteq X$ be the schematic image of $f$ \cite[Tag 01R7]{StaProj}.
The points in $|Z|$ are those points that specialize from a point in $f(\Spec K)$ \cite[Tag 02JQ]{StaProj}.
For every $x\in \Spec A$, we consider the sequence of maps \[\Spec \on{Frac}(A_x)\to \Spec K_x\to Z_x\to X_x\to \Spec A_x,\]
 obtained from localizing at the local ring of $x$.
 It is still true that $Z_x$ is the schematic image of $\Spec \on{Frac}(K_x)\to X_x$ \cite[Tag 081I]{StaProj}. 
 By the definition of the scheme theoretic image, for every $z\in |Z|$ mapping to $x$, we get injective maps of local rings 
 \[A_x\to \calO_{Z_x,z}\to K_x\to  \on{Frac}(A_x).\] 
 The map $A_x\to \calO_{Z_x,z}$ is necessarily an isomorphism since $A_x$ is a valuation ring and $\calO_{Z_x,z}\subseteq \on{Frac}(A_x)$ dominates $A_x$.
The inverse of this isomorphism provides a lift of $\Spec A_x\to Z_x\to X_x$, and by uniqueness of the lift we know that there is exactly one point $z\in |Z_x|$ with image $x$. 
This shows that $|Z|\to |\Spec A|$ is a homeomorphism inducing isomorphisms of local rings.
This suffices to prove that $Z\to \Spec A$ is an isomorphism of schemes.
\end{proof}

\begin{proposition}
	\label{valuativepreki}
	Suppose that $f:X\to Y$ is a perfectly finitely presented proper map of perfect schemes. Then $X^{\diamond/\circ\circ}\to Y^{\diamond/\circ\circ}$ is weakly partially proper. 
\end{proposition}
\begin{proof}
	The valuative criterion of partial properness is v-local on the target (see \Cref{partial proper v-local}), so we may assume that $Y=\Spec A$ is affine.
	Moreover, it suffices to test the criterion on a basis for the v-topology.
	Let $\Spa (R,R^+)$ be a strictly totally disconnected space. 
	Now, $X^{\diamond/\circ\circ}(R,R^+)=X(\Spec R^+_\red)$ since $\Spa( R,R^+)$ splits all open covers. 
	Analogously, $Y^{\diamond/\circ\circ}(R,R^+)=Y(\Spec R^+_\red)$.
	Fix a diagram
	\begin{center}
	\begin{tikzcd}
	 \Spec R^\circ_\red \arrow{r} \arrow{d}  & X \arrow{d} \\
 \Spec R^+_\red \arrow{r} \ar[dashed]{ru} & Y 
	\end{tikzcd}
	\end{center}
	It follows from \Cref{enhancedporperness}, that one can find a unique solution to the dashed arrow above.
	Indeed, the connected components of $\Spec R^+_\red$ are the spectrum of a valuation ring, $\Spec R_\red^\circ\subseteq \Spec R_\red^+$ is pro-open and it meets every connected component.
%
\end{proof}

    Weakly separated maps satisfy, almost by definition, the valuative criterion of separatedness.

\begin{proposition}
	\label{partiallyseparatedisvaluativelyseparated}
	Let $X\to Y$ be a weakly separated map of small v-sheaves and consider the commutative diagram
	\begin{center}
	\begin{tikzcd}
		\Spa(R,R^\circ) \arrow{r} \arrow{d}  & X \arrow{d} \\
		\Spa(R,R^+) \arrow{r} \arrow[dashed]{ru} & Y.
	\end{tikzcd}
	\end{center}
	There exist at most one dashed arrow making the diagram commutative.
\end{proposition}
\begin{proof}
	Suppose that we have a map $(f_1,f_2):\Spa(R,R^+)\to X\times_Y X$ such that $\Spa(R,R^\circ)\to X\times_Y X$ factors through the relative diagonal $\Delta_{X/Y}:X\to X\times_Y X$. 
	Since the diagonal is partially proper, there is a unique map $\Spa(R,R^+)\to X$ extending the map $\Spa(R,R^\circ)$ and whose composition with $\Delta_{X/Y}$ gives $(f_1,f_2)$. This proves $f_1=f_2$.
\end{proof}

\begin{proposition}
	\label{standardpropertiesweaklysp}
	Let $f:X\to Y$, $g:Y\to Z$ and $q:\widetilde{Y}\to Y$ be maps of v-sheaves. Let $h=g\circ f$, let $\tilde{f}:\widetilde{X}\to \widetilde{Y}$ the base change of $f$ along $q$. The following hold.
	\begin{enumerate}
		\item If $f$ and $g$ are weakly separated, then $h$ is weakly separated.
		\item If $f$ is weakly separated, then $\tilde{f}$ is weakly separated.
		\item If $q$ is a v-cover and $\tilde{f}$ is weakly separated, then $f$ is weakly separated.
		\item If $h$ is weakly separated and $g$ is weakly separated, then $f$ is weakly separated.
	\end{enumerate}
\end{proposition}
\begin{proof}
	Recall that being partially proper is stable under base change, composition and can be verified v-locally on the target.

	For the first statement consider the map $X\to X\times_Y X\to X\times_Z X$, it suffices to prove $X\times_Y X\to X\times_Z X$ is partially proper, but it comes from base change of $X\times_Z X\to Y\times_Z Y$ along $Y\to Y\times_Z Y$ which is partially proper.

	For the second and third statements, note that $\widetilde{X}\to \widetilde{X}\times_{\widetilde{Y}} \widetilde{X}\to \widetilde{Y}$ is the base change of $X\to X\times_Y X\to Y$ along $q$ (see also \Cref{partial proper v-local}). 

	For the fourth statement, note that $X\times_Y X\to X\times_Z X$ is injective and in particular separated. 
	Since $X\to X\times_Z X$ is assumed to be partially proper, by cancellation then $X\to X\times_Y X$ also is.
\end{proof}

\begin{proposition}
	\label{standardpropertiesweaklypar}
	Let the notation be as in \Cref{standardpropertiesweaklysp}. The following hold.
	\begin{enumerate}
		\item If $f$ and $g$ are weakly partially proper, then $h$ is weakly partially proper.
		\item If $f$ is weakly partially proper, then $\tilde{f}$ is weakly partially proper.
		\item If $q$ is a v-cover and $\tilde{f}$ is weakly partially proper, then $f$ is weakly partially proper.
		\item If $h$ is weakly partially proper and $g$ is weakly separated, then $f$ is weakly partially proper.
	\end{enumerate}
\end{proposition}
\begin{proof}
	For the first statement, fix compatible maps $t:\Spa(R,R^+)\to Z$ and $\Spa(R,R^\circ)\to X$, we can lift $t$ first to $Y$ and then to $X$.
	The second statement follows from the universal property of base change.

	For the third statement we fix compatible maps $\Spa(R,R^\circ)\to X$ and $\Spa(R,R^+)\to Y$. 
	We fix a v-cover $\Spa(T,T^+)\to \Spa(R,R^+)$ such that the map $\Spa(T,T^+)\to Y$ lifts to $\Spa(T,T^+)\to \widetilde{Y}$ and we let $\Spa(S,S^+)=\Spa(R,R^+)\times_{\Spa(T,T^+)} \Spa(R,R^+)$. 
	We consider the commutative diagram.
	\begin{center}
	\begin{tikzcd}
	& &\Spa(R,R^\circ) \ar{rd} & &	\\
		\Spa(S,S^\circ) \arrow[shift right]{r} \arrow[shift left]{r} \arrow{d}  &	\Spa(T,T^\circ) \arrow{ru} \arrow{r}{o} \arrow{d}  &\widetilde{X} \arrow{r} \arrow{d}  &  X \arrow{d} \\
		\Spa(S,S^+) 	\arrow[shift right]{r} \arrow[shift left]{r} 	& \Spa(T,T^+) \ar[dashed]{ru}{l} \ar{rd} \arrow{r}{o^+} & \widetilde{Y}  \arrow{r} & Y\\
										& &\Spa(R,R^+) \ar[dashed]{ruu}{e} \ar{ru} & &	
	\end{tikzcd}
	\end{center}
	The existence of the dashed arrow labeled $e$ follows from the existence of the dashed arrow labeled $l$. Indeed, since $X$ is weakly separated and the two maps $\Spa(S,S^\circ)\to X$ agree, then the two maps $\Spa(S,S^+)\to X$ induced by $l$ also agree and define descent datum for a map $\Spa(R,R^+)\to X$.

	For the fourth statement, note that $Y\to Z$ is by assumption weakly separated so $Y\to Y\times_Z Y$ is partially proper and in particular weakly partially proper. 
	By cancellation $X\to Y$ must be weakly partially proper.
\end{proof}

In general, being valuative cannot be verified v-locally. 
That is, there are maps of prekimberlites $f:Y\to X$ such that $Y$ is valuative, $f$ is a v-cover, but $X$ is not valuative. 
We have the following special case.

    \begin{proposition}
	    \label{valuativityisappropriatelyv-local}
    Let $f:X\to Y$ be a map of prekimberlites. Suppose that $Y$ is valuative and that there is a valuative prekimberlite $Z$ and a v-cover $g:Z\to Y$ such that $W=X\times_Z Y$ is valuative, then $X$ is valuative. 
    \end{proposition}
    \begin{proof}
	    We fix a commutative diagram and we try to solve for the dashed arrow
	    \begin{equation}
		    \label{dashedarrowlift}
	    \begin{tikzcd}
		    \Spa(R,R^\circ) \arrow{r}{o_X} \arrow{d}  & X \arrow{d} \\
		    \Spa(R,R^+) \arrow{r}{o_X^H} \ar[dashed]{ru} & X^{\on{H}}. 
	    \end{tikzcd}
	    \end{equation}
	    Since $Y$ is valuative we get a unique map $o^+_Y:\Spa (R,R^+)\to Y$ lifting the natural map $o_Y:\Spa(R,R^\circ)\to Y$.
	    Since $Z\to Y$ is surjective we may find a v-cover $c_1:\Spa(T_1,T_1^+)\to \Spa (R,R^+)$ and we can choose a map $e_{Z}^+:\Spa(T_1,T_1^+)\to Z$ lifting the natural map $o^+_Y\circ c_1:\Spa(T_1,T_1^+)\to Y$.
	    We also have, by construction, a map $o_X\circ c_1^\circ:\Spa(T_1,T_1^\circ)\to X$ whose induced map $f\circ o_X\circ c_1^\circ$ agrees with $g\circ e_Z^+$ when the latter map is restricted to $\Spa(T_1,T_1^\circ)\subseteq   \Spa(T_1,T_1^+)$.
	    This gives a map $\Spa(T_1,T_1^\circ)\to W$ and since $W$ is valuative we also get a map $\Spa(T_1,T_1^+)\to W$. 
	    Indeed, the map $\Spa(T_1,T_1^+)\to W^{\on{H}}$ is well-defined since $W^{\on{H}}=X^{\on{H}}\times_{Z^{\on{H}}} Y^{\on{H}}$.
	    Projecting to $X$ we get a map $\Spa(T_1,T_1^+)\to X$ lifting $o_X\circ c_1^\circ$ and compatible with $\Spa(T_1,T_1^+)\to X^{\on{H}}$.
	    By separatedness of $X\to X^{\on{H}}$, this lift is unique and did not depend of the choice of $e^+_{Z}$.
	    This shows that the dashed arrow in the diagram
	    \begin{center}
	    \begin{tikzcd}
		    \Spa(T_1,T_1^\circ) \arrow{r} \arrow{d}  & X \arrow{d} \\
		    \Spa(T_1,T_1^+) \arrow{r}\ar[dashed]{ur} & X^{\on{H}} 
	    \end{tikzcd}
	    \end{center}
	    can be uniquely lifted. 
	    Moreover, if $\Spa(T_2,T_2^+)=\Spa(T_1,T_1^+)\times_{\Spa(R,R^+)}\Spa(T_1,T_1^+)$ 
	    we get a commutative diagram
	    \begin{center}
	    \begin{tikzcd}
		    \Spa(T_2,T_2^\circ) \ar[shift left]{r} \ar[shift right]{r} \ar{d}	&    \Spa(T_1,T_1^\circ) \arrow{r} \arrow{d} & \Spa(R,R^\circ) \ar{d} \ar{r}  & X \arrow{d} \\
		    \Spa(T_2,T_2^+) \ar[shift left]{r} \ar[shift right]{r} 	& \Spa(T_1,T_1^+) \arrow{r} \ar{rru} & \Spa(R,R^+) \ar{r}  & X^{\on{H}},
	    \end{tikzcd}
	    \end{center}
	    and since $X(R,R^+)=\on{Eq.}(X(T_1,T_1^+)\rightrightarrows X(T_2,T_2^+))$, the arrow in the diagram \eqref{dashedarrowlift} can be uniquely constructed. 
    \end{proof}

Just as a separated v-sheaf admits a canonical compactification, which is a canonical way to turn a separated v-sheaf into a partially proper one \cite[Proposition 18.6]{Sch17}, an analogous statement holds for weakly separated v-sheaves.
We discuss slight generalizations of some statements proved in \cite[$\mathsection 18$]{Sch17}. 
\begin{definition}{\cite[Proposition 18.6]{Sch17}}
Suppose that $Y$ is a weakly separated v-sheaf, we let $\overline{Y}$ be the v-sheafification of the formula 
\[\overline{Y}^{\on{pre}}(R,R^+)=Y(R,R^\circ).\] 
We call $\overline{Y}$ the \textit{canonical compactification} of $Y$. 
\end{definition}
\begin{remark}
	Note that if $Y$ is a separated v-sheaf, then $\overline{Y}$ agrees with the usual canonical compactification of \cite[$\mathsection 18$]{Sch17}.	
\end{remark}
\begin{proposition}
	\label{evaluateonstrictlytotally}
	If $Y$ is a weakly separated v-sheaf and $\Spa( R,R^+)$ is totally disconnected, then $\overline{Y}(R,R^+)=\overline{Y}^{\on{pre}}(R,R^+)$.
\end{proposition}
\begin{proof}
	The same proof as in \cite[Proposition 18.6]{Sch17} works with the role of \cite[Proposition 10.10]{Sch17} loc. cit.~ replaced by \Cref{partiallyseparatedisvaluativelyseparated}. 
\end{proof}
Note that the construction $Y\mapsto \overline{Y}^{\on{pre}}$ is functorial and commutes with arbitrary limits of v-sheaves. 
By \Cref{evaluateonstrictlytotally}, the same can be said about $Y\mapsto \overline{Y}$.
\begin{proposition}
	If $Y$ is a weakly separated v-sheaf, then $\overline{Y}$ is weakly partially proper and the map $Y\to \overline{Y}$ is initial among maps $Y\to Z$ with $Z$ a weakly partially proper v-sheaf. 
\end{proposition}
\begin{proof}
	The diagonal $\overline{Y}\to \overline{Y}\times \overline{Y}$ is injective and consequently separated. 
For a weakly separated map to satisfy the valuative criterion of partial properness it suffices to verify it on a basis for the v-topology.
On a totally disconnected space $\Spa(R,R^+)$ we have that $\overline{Y}^{\on{pre}}(R,R^+)=\overline{Y}(R,R^+)=Y(R,R^\circ)$. 
This allows us to verify directly first that $\Delta_{\overline{Y}}$ and then that $\overline{Y}\to \ast$ satisfy the left lifting property of maps with source $\Spa(R,R^\circ)\to \Spa(R,R^+)$.

Moreover, if $Y\to Z$ is a map and $Z$ is weakly partially proper, then we get a map $\overline{Y}\to \overline{Z}=Z$.
On the other hand, given a map $\overline{Y}\to Z$ we can restrict it to a map $Y\to Z$. 
These operations are inverses of each other.
\end{proof}
	
If $Y\to Z$ is a map of v-sheaves that are weakly separated, then $Y\to Z$ is weakly separated, and we can define a relative compactification $\overline{Y}^{/Z}=\overline{Y}\times_{\overline{Z}}Z$.
In this case, $Y\to \overline{Y}^{/Z}\to Z$ is initial among maps $Y\to W\to Z$ with $W\to Z$ weakly partially proper. 
\begin{definition}
	We say that a weakly separated map of v-sheaves $Y\to Z$ is \textit{weakly compactifiable} if it can be factored as $Y\to W\to Z$ where $W\to Z$ is weakly partially proper and $Y\to W$ is an open immersion. 
\end{definition}
\begin{remark}
Note that a separated map of v-sheaves is compactifiable if and only if it is weakly compactifiable.  	
\end{remark}
\begin{proposition}
	A weakly separated map of v-sheaves $Y\to Z$ is weakly compactifiable if and only if $Y\to \overline{Y}^{/Z}$ is an open immersion.
\end{proposition}
\begin{proof}
	Let $Y\to W\to Z$ be a factorization, by the universal property of $Y\to \overline{Y}^{/Z}$ we get a map $\overline{Y}^{/Z}\to W$ which we claim is an injection. 
	Indeed, it suffices to show $\overline{Y}^{/Z}(R,R^\circ)\to W(R,R^\circ)$ is injective since $\overline{Y}^{/Z}\to W$ is weakly partially separated. 
	But the map $\overline{Y}^{/Z}(R,R^\circ)\to W(R,R^\circ)$ identifies with the map $Y(R,R^\circ)\to W(R,R^\circ)$ which is injective.
	Since $Y\to W$ is an open immersion $Y\to \overline{Y}^{/Z}$ is also an open immersion.
\end{proof}
\begin{proposition}
	\label{compositionofcompacti}
	Let $f:X\to Y$ and $g:Y\to Z$ be maps of weakly separated v-sheaves and let $h=g\circ f$. The following hold.
	\begin{enumerate}
		\item If $f$ and $g$ are weakly compactifiable, then $h$ is weakly compactifiable. 
		\item If $h$ is weakly compactifiable, then $f$ is weakly compactifiable. 
	\end{enumerate}
\end{proposition}
\begin{proof}
	Consider the following commutative diagram with Cartesian squares
\begin{center}
\begin{tikzcd}
	X \ar{r}{e_1} \ar{rd}{e_2}	& \overline{X}^{/Y} \ar{r} \ar{d}{e_3} &Y \ar{d}{e_4} & \\
	& 	\overline{X}^{/Z}	\arrow{r} \arrow{d}  & \overline{Y}^{/Z} \ar{r} \ar{d} \arrow{d} & Z \ar{d}  \\
	& \overline{X}\arrow{r} & \overline{Y} \ar{r} & \overline{Z}. 
\end{tikzcd}
\end{center}
If $e_1$ and $e_4$ are open immersions, then $e_3$ and $e_2$ are open immersions. 
Conversely, if $e_2$ is an open immersion, then $e_1$ is an open immersion since $e_3$ is an injection.
\end{proof}

\begin{corollary}
	If $f:X\to Y$ is a map of valuative prekimberlites and $X^\red\to Y^\red$ is perfectly finitely presented and separated, then $X\to Y$ is compactifiable.
\end{corollary}
\begin{proof}
	By \Cref{compositionofcompacti}, it suffices to show that the map $f^{\on{H}}:X^{\on{H}}\to Y^{\on{H}}$ is weakly compactifiable. 
	Indeed, $f^{\on{H}}\circ{\on{SP}}_X$ would be weakly compactifiable and since $f^{\on{H}}\circ{\on{SP}}_X={\on{SP}}_Y\circ f$, \Cref{compositionofcompacti} above shows that $f$ would also be weakly compactifiable.
	Since $f$ is already separated, we would conclude that it is compactifiable.

	Now, by Nagata's theorem \cite[Tag 0F41, 0ATT]{StaProj}, we can find a perfectly finitely presented compactification $X^\red\subseteq S\to Y^\red$. 
	This induces an open immersion $X^{\on{H}}\to S^{\diamond/\circ \circ}$ and by \Cref{valuativepreki} a weakly partially proper map $S^{\diamond/\circ\circ}\to Y^{\on{H}}$.
\end{proof}

\subsection{Kimberlites}
Now that we have discussed valuative prekimberlites, we discuss kimberlites.

    \begin{definition}{\cite[Remark 4.37, Definition 4.35]{Gle24}}
	    \label{definitionofkimberlite}
	    Let $X$ be a valuative prekimberlite.
	    \begin{enumerate}
		    \item We say that $X$ is an \textit{affine kimberlite} if $X^\red$ is an affine scheme and $X^{\on{an}}$ is a spatial diamond.
		    \item We say that $X$ is a kimberlite if for all open affine subschemes, $U\subseteq X^\red$, the induced Zariski open formal neighborhood $\widehat{X}_{/U}$ is an affine kimberlite.
	    \end{enumerate}
    \end{definition}

    \begin{proposition}
	    \label{twodefinitionsofkimberlite}
	    Let $X$ be a valuative prekimberlite. 
	    Suppose that $X^{\on{an}}$ is a quasiseparated\footnote{As shown in \Cref{quasiseparatednessofprekimberlites}, the quasiseparatedness is automatic. We leave it as it is to point out that this is a subtlety implicitly used in our argument.} locally spatial diamond and consider the topological specialization map $\on{sp}_{X^{\on{an}}}:|X^{\on{an}}|\to |X^\red|$. The following hold.
	    \begin{enumerate}
		    \item $X$ is a kimberlite if and only if $\on{sp}_{X^{\on{an}}}$ is quasicompact. 
		    \item If $X$ is a kimberlite, then the map is spectral and closed. 
	    \end{enumerate}
    \end{proposition}
    \begin{proof}
	    The hypothesis say that the pair $(X,X^\an)$ is a smelted kimberlite as in \cite[Definition 4.35]{Gle24}.
	    By \cite[Theorem 4.40]{Gle24}, the topological specialization map is automatically specializing and a spectral map of locally spectral spaces.
	    Whether $\on{sp}_{X^\an}$ is quasicompact or not can be tested on an open cover of $|X^\red|$. 
	    Fix an open cover of $X^\red$ by affine subschemes 
	    \[\coprod_{i\in I} U_i\to X.\]
	    Then $|(\widehat{X}_{/U_i})^\an|={\on{sp}}_{X^\an}^{-1}(U_i)$.
	    If $X$ is a kimberlite, then by hypothesis $(\widehat{X}_{/U_i})^\an$ is a spatial diamond and consequently quasicompact \cite[Proposition 11.18.(i)]{Sch17}. 
	    Conversely, if ${\on{sp}}_{X^\an}$ is quasicompact, then $|\widehat{X}_{/U_i})^\an|\subseteq |X^\an|$ is a quasicompact and quasiseparated topological space. 
	    By \cite[Proposition 11.19.(iii)]{Sch17}, $(\widehat{X}_{/U_i})^\an$ is a spatial diamond as we wanted to show. 
	    The second statement is \cite[Theorem 4.40]{Gle24}.
    \end{proof}

\begin{remark}
    That the specialization map is continuous for both the usual and the constructible topology, is a key technical ingredient in the computation of the connected components of affine Deligne--Lusztig varieties and moduli spaces of shtukas \cite{Gle21}, \cite{gleason2022connected}.
\end{remark}

    \begin{proposition}{\cite[Proposition 4.41.(4)]{Gle24}}
	    \label{cite prop 4.41}
    If $X$ is a kimberlite and $Z\to X$ is a formally adic closed immersion, then $Z$ is also a kimberlite.
    \end{proposition}

    Recall that, in the theory of v-sheaves, being quasiseparated is an absolute notion which disagrees with the relative notion of being quasiseparated over the final object. 
    In particular, there are plenty of v-sheaves separated over $\ast$ that are not quasiseparated in an absolute sense.
    To wit, $\ast$ itself is not quasiseparated.
    By definition, prekimberlites are separated over $\ast$. A consequence of this is that the analytic locus of a prekimberlite is always quasiseparated. 

    \begin{proposition}
	    \label{quasiseparatednessofprekimberlites}
    Let $X$ be a prekimberlite, then $X^\an$ is quasiseparated. 	
    \end{proposition}
    \begin{proof}
	    Let $\Spa (R_1,R^+_1)\to X^\an$ and $\Spa( R_2,R^+_2)\to X^\an$ be two maps with $\Spa (R_i,R_i^+)\in\Perf$, we must show that 
	    \[W:=\Spa (R_1,R^+_1)\times_{X^\an} \Spa (R_2,R^+_2)\] is quasicompact.	
    After replacing $R_1$ and $R_2$ by a v-cover we can find formally adic formalizations $\Spd R^+_1\to X$ and $\Spd R_2^+\to X$, and we may write $W$ as the analytic locus in $V:=\Spd R^+_1\times_X \Spd R_2^+$.
    Furthermore, we may rewrite $V$ as the base change of the map $\Spd R_1^+\times \Spd R_2^+\to X\times X$ along the diagonal map $\Delta: X\to X\times X$.
    In this way, we get a formally adic closed immersion $V\to \Spd R_1^+\times \Spd R_2^+$ where the target is a formal v-sheaf.  
By \Cref{cite prop 4.41}, $V$ is a kimberlite and $W=V^\an$ is a spatial diamond.
    In particular, it is quasicompact.
    \end{proof}

    We have the following stability property.
    \begin{proposition}{\cite[Proposition 4.42]{Gle24}}
    If $X$ is a kimberlite and $Y\to X$ is an \'etale formal neighborhood, then $Y$ is a kimberlite. 
    \end{proposition}

    We also have the following useful lemma which was implicitly proved in \cite[Proposition 2.34]{AGLR22}. We recall the proof for the convenience of the reader, since the setup is slightly different.
    
    \begin{lemma}
	    \label{annoyinglemmaonsurjectivityofspecialization}
	    Let $X$ be an affine kimberlite. 
	    Let $Z\subseteq X^\red$ denote the closed subscheme with $|Z|=\on{sp}(X^{\on{an}})$.
	    Suppose that $f:\Spa (R,R^+)\to X^{\on{an}}$ is a v-cover and that $X$ formalizes $f$. 
	    Then $|\Spd R^+_\red|\to |Z^{\diamond}|$ is surjective. 
	    In particular, $\Spd R^+_\red\to Z^\diamond$ is a v-cover.  
    \end{lemma}
    \begin{proof}
	    Let $X^\red=\Spec A$ and let $z\in |Z^{\diamond}|\subseteq |X^{\on{Red}}|=|\Spd(A,A)|$.
	    To $z$ we can attach a support ideal $\frakp_z$ and a specialization ideal $\frakq_z$ with $\frakp_z\subseteq\frakq_z\subseteq A$ and with $\frakp_z, \frakq_z\in Z$ \cite[Definition 2.4]{Gle24}.
	    There is a perfect valuation ring $V$, and a map $\Spec V \to Z$ whose special point maps to $\frakq_z$ and whose generic point maps to $\frakp_z$.
	    Moreover, $z$ is in the image of the induced map $|\Spec V^\diamond| \to |Z^\diamond|$. 

	    Since $|Z|=\on{sp}(X^{\on{an}})$ there is a formalizable geometric point $\Spa(C,C^+)\to X^\an$ such that the closed point of $\Spec C^+_\red$ maps to $\frakp_z$ under the map $(\Spd C^+)^\red\to X^\red$.
	    Since the closed point of $\Spec C^+_\red$ maps to the same point as the generic point of $\Spec V$, we can compose the valuations to obtain a valuation ring $C^+_{\red,V}\subseteq C^+_\red$ and a map $\Spec C^+_{\red,V}\to \Spec A$ such that $\Spec C^+_{\red,V}/\frakp_z\to \Spec A$ factors through $\Spec V$ and gives an extension of valuation rings with $(\Spec C^+_{V,\red})_{(\frakp_z)}=\Spec C_\red^+$. 
	    Let $C^+_V\subseteq C^+$ denote the unique valuation subring containing $C^{\circ \circ}$ and such that $(C^+_V)_\red=C^+_{V,\red}$.
	    We obtain a commutative diagram 
	    \begin{center}
	    \begin{tikzcd}
		    \Spa(C,C^+) \arrow{r} \arrow{d}  & X^\an \arrow{r} & X \ar{d}{\on{SP}} \\
		    \Spa(C,C^+_V) \arrow{rr} & & X^{\on{H}}.
	    \end{tikzcd}
	    \end{center}
	    By valuativity of $X$, this lifts to a map $\Spa(C,C^+_V)\to X^{\on{an}}$.
	    Replacing $(C,C^+)$ by a v-cover we may assume that the map factors through a map $\Spa(C,C^+_V)\to \Spa (R,R^+)$.
	    The formalization $\Spd C^+_V\to X$ induces a map $\Spec C^+_{V,\red}\to Z$ that factors (and surjects onto) $\Spec V$.
	    In particular, $z$ is in the image of $|\Spd C^+_{V,\red}|\to |\Spd R^+_\red|\to |X|$.
	    Which shows that $|\Spd R^+_\red|\to |Z^\diamond|$ is surjective.
	    The final statement follows from \cite[Lemma 2.26]{Gle24} combined with \cite[Lemma 12.11]{Sch17}.
    \end{proof}

    The theory of kimberlites developed in \cite{Gle24}, \cite{GL22}, \cite{AGLR22} and \cite{GI_Bunmer} is sufficient for many purposes. 
    Nevertheless, it falls short in answering the following natural foundational questions.\footnote{We have not found a proof or a counterexample to the statements in \Cref{naturalquestions}. 
    We suspect that in all the cases a counterexample should exist.} 
Spatial kimberlites resolve these issues. 
\begin{question}
	\label{naturalquestions}
Let $X$ be and $Y$ be kimberlites.	
\begin{enumerate}
	\item Is $X\to \ast$ representable in locally spatial diamonds?
	\item Is $X\times Y$ a kimberlite?
	\item Is $X$ pro-\'etale locally formalizing?
	\item If $S\subseteq X^\red$ is constructible, is $\widehat{X}_{/S}$ a kimberlite?
\end{enumerate}
\end{question}

Our motivation to define spatial kimberlites is to deduce \Cref{naturalquestions}, particularly \Cref{naturalquestions}.(4), from the stronger axioms. 

    \section{Spatial Kimberlites}


    \begin{definition}
	    \label{affinespatialdefi}
	    We say that an affine kimberlite $X$ is \textit{spatial} if the following conditions hold.
	    \begin{enumerate}
		    \item $X$ formalizes geometric points. 
		    \item There is a qcqs formally adic v-cover $f:Y\to X$, where $Y$ is an affine formal v-sheaf (i.e. $Y=\Spd B$ for $B$ an $I$-adic ring).  
	    \end{enumerate}
    \end{definition}

    The following proposition explains the relevance of the axiom of \Cref{affinespatialdefi}.(1).

    \begin{proposition}
	    \label{prekimbproetaleformalizing} 
	    Let $X$ be a prekimberlite that formalizes geometric points. 
	    If $S=\Spa (R,R^+)$ is a strictly totally disconnected perfectoid space and $f:S\to X$ is a map, then $X$ formalizes $f$.
    \end{proposition}
    \begin{proof}
    Since strictly totally disconnected perfectoid spaces split open covers, we may assume that $X$ is an affine prekimberlite.	
    Consider the graph morphism $S\to X\times \Spd(R^+)$. 
    Let $S':=\Spa (R',R'^+)\to S$ be a v-cover that formalizes over $X$ and let $\calF$ denote the sheaf theoretic image of the map $\Spd(R'^+)\to X\times \Spd R^+$. 
    We may choose $S'$ to also be strictly totally disconnected.
    We wish to prove that the projection map $\calF\to \Spd(R^+)$ is an isomorphism.
   Consider the following commutative diagram:
   \begin{equation}
	   \label{basic diagram whih we basechange}
   \begin{tikzcd}
	   S' \ar{r}\ar{d} & \Spd R'^+ \ar{r}\ar{d}& X\times \Spd R'^+ \ar{r}\ar{d}& \Spd R'^+ \ar{d}\\	
	   S \ar{r} & \calF \ar{r} & X\times \Spd R^+ \ar{r} & \Spd R^+ \\	
   \end{tikzcd}	
   \end{equation}
   Since the map $\Spd R'^+\to \Spd R^+$ is surjective, the map $\calF\to \Spd R^+$ is also surjective.
   We now prove it is also injective.
   First, the map $\Spd R^+\times X\to \Spd R^+$ is quasiseparated since $X\to \ast$ is separated and since $\calF\subseteq X\times \Spd R^+$, then $\calF$ is also qs over $\Spd R^+$.
   Moreover the map $\Spd R'^+\to \Spd R^+$ is qcqs which implies that $\calF$ is qcqs over $\Spd R^+$.
   By \cite[Lemma 12.5]{Sch17}, we can check bijectivity of $\calF\to \Spd R^+$ on geometric points. 
   Any geometric point factors through a connected component of $\Spd R^+$. 
   Let $x\in \pi_0(\Spd R^+)$.
   Then $\calF_x$ is the sheaf-theoretic image of the map $\Spd R'^+_x\to X\times \Spd C_x^+$ all over $\Spd C_x^+$.
   Here $\Spd R'^+_x= \Spd R'^+\times_{\Spd R^+} \Spd C_x^+$. 
   This gives rise to the following commutative diagram which is simply the basechange of \eqref{basic diagram whih we basechange} along the closed immersion $\Spd C_x^+\to \Spd R_x^+$.
   \begin{equation}
   \begin{tikzcd}
	   \Spa(R'_x,R'^+_x) \ar{r}\ar{d} & \Spd R_x'^+ \ar{r}\ar{d}& X\times \Spd R_x'^+ \ar{r}\ar{d}& \Spd R_x'^+ \ar{d}\\	
	   \Spa(C_x,C_x^+) \ar{r} & \calF_x \ar{r} & X\times \Spd C_x^+ \ar{r} & \Spd C_x^+ \\	
   \end{tikzcd}	
   \end{equation}
   By hypothesis, the second projection map $X\times \Spd C_x^+\to \Spd C_x^+$ has a unique section $\Spd C_x^+\to X\times \Spd C_x^+$ compatible with the graph map $\Spd C_x\to X\times \Spd C_x$ induced by $f$.
   This section induces a formalization 
   \[\Spd R'^+_x\to \Spd C_x^+\to X\times \Spd C_x^+\] 
   of the map $\Spa (R_x',R'^+_x)\to X\times \Spa C^+_x$.
   Since formalizations over $X\times \Spd C_x^+$ are unique, we conclude that the map $\calF_x\to X\times \Spd C_x^+$ factors through the section $\Spd C_x^+\to X\times \Spd C_x^+$. 
   This proves that $\calF_x\to \Spd C_x^+$ is injective and surjective, consequently it induces a bijection of geometric points. 
    \end{proof}

    \begin{proposition}
	    \label{affineislocallyspatial}
	    Suppose that $X$ is an affine spatial kimberlite and $U\subseteq X^\red$ is an affine open subset, then $\widehat{X}_{/U}$ is an affine spatial kimberlite.	
    \end{proposition}
    \begin{proof}
	    By assumption, $\widehat{X}_{/U}\to X$ is a formally adic open immersion. 
	    Consequently, a map $\Spd C^+\to X$ factors through $\widehat{X}_{/U}$ if and only if $\Spa(C,C^+)$ does. 
	    In particular, $\widehat{X}_{/U}$ formalizes geometric points.   
	    If $\Spd B\to X$ is a qcqs formally adic v-cover, then 
	    $\widehat{X}_{/U}\times_X \Spd B$ is again an affine formal v-sheaf and $\widehat{X}_{/U}\times_X \Spd B\to \widehat{X}_{/U}$ is again a qcqs formally adic v-cover.
    \end{proof}

    \begin{definition}
	    \label{spatial-kimberlite-defi}
	    Let $X$ be a kimberlite.  
    \begin{enumerate}
    \item We say that $X$ is a \textit{locally spatial kimberlite} if for all affine open neighborhoods $U\subseteq |X^\red|$ the associated open formal neighborhood $\widehat{X}_{/U}$ is an affine spatial kimberlite. 
	    \item We say that $X$ is a \textit{spatial kimberlite} if it is a locally spatial kimberlite and $|X^\red|$ is qcqs.
    \end{enumerate}
    \end{definition}
    It follows from \Cref{affineislocallyspatial} that affine spatial kimberlites are spatial kimberlites.
    \Cref{affinecommunication} below shows that we can check being spatial on an affine open cover. 
    In particular, spatial kimberlites that are affine prekimberlites coincide with affine spatial kimberlites.

    \begin{proposition}
	    \label{affinecommunication}
	    Suppose that $X$ is an affine prekimberlite, and that there is an affine open cover $\coprod_{i\in I} U_i\to X^\red$ such that each $\widehat{X}_{/U_i}$ is an affine spatial kimberlite, then $X$ is an affine spatial kimberlite.	
    \end{proposition}
    \begin{proof}
    That $X$ formalizes geometric points is clear since geometric points factor through $\widehat{X}_{/U_i}$ for some $i\in I$ and formalize over it.
    Since $X^\red$ is affine we may assume that $I$ is finite.
    By \cite[Theorem 4.27]{Gle24}, the map $\coprod_{i=1}^n \widehat{X}_{/U_i}\to X$ is a qcqs formally adic v-cover. 
    By hypothesis, there are formally adic qcqs v-covers $\Spd B_i\to \widehat{X}_{/U_i}$ with each $\Spd B_i$ an affine formal v-sheaf.
    Then $\coprod_{i=1}^n \Spd B_i\to X$ is a qcqs formally adic v-cover and the source is still an affine formal v-sheaf. 
    \end{proof}

    \subsection{Being spatial is \'etale local.}

    \begin{proposition}
	    \label{spatialkimberliteisetalelocal}
    Let $f:X\to Y$ be a map of prekimberlites. The following hold:
    \begin{enumerate}
	    \item If $f$ is formally adic \'etale and $Y$ is a locally spatial kimberlite, then $X$ is a locally spatial kimberlite.
	    \item If $f$ is a formally adic \'etale cover and $X$ is a locally spatial kimberlite, then $Y$ is a locally spatial kimberlite.
    \end{enumerate}
    \end{proposition}
    \begin{proof}
	    The valuative property is stable under formal \'etale neighborhoods and can be checked after an \'etale formal v-cover, so we may assume $X$ and $Y$ are valuative.
	    For the first statement, by \Cref{affinecommunication} we can assume $X$ is an affine prekimberlite and that the map $X\to Y$ factors through $V=\widehat{Y}_{/V^\red}$ for $V^\red\subseteq Y^\red$ an affine open subset.
	     By \cite[Proposition 4.42.(2)]{Gle24}, $X=\widehat{Y}_{/X^\red}$ is a kimberlite. 
	    Let $\Spd(C^+)\to V$ be the formalization of a geometric point, and let $P=\Spd(C^+)\times_V X$.
	    By \cite[Corollary 4.28]{Gle24}, $P$ is of the form $\coprod_{i\in I}\Spd(C_i^+)$ where $C_i^+=C^+[\frac{1}{c}]$ for some element $c\in C^+$. 
	    This shows that $X$ formalizes geometric points, since $V$ does.
	    Let $W:=\Spd(B)\to V$ be a qcqs formally adic v-cover and let $\widetilde{W}=W\times_V X$. 
	    The map $\widetilde{W}\to W$ is formally adic \'etale and by \cite[Corollary 4.28]{Gle24}, $\widetilde{W}$ is an affine formal v-sheaf. 
	    Moreover, $\widetilde{W}\to X$ is again a qcqs formally adic v-cover, which shows that $X$ is an affine spatial kimberlite. 

	    Conversely, assume that $f$ is a formally adic \'etale cover and that $X$ is locally spatial. 
	    Replacing $Y$ by an affine neighborhood and $X$ by an affine cover $\coprod_{i\in I} \widehat{X}_{/U_i}$, and then choosing finite subset of it, we may assume without losing generality that $X$ is an affine spatial kimberlite and that $Y$ is an affine prekimberlite.

	    Now, $Y^{\on{an}}$ is a spatial diamond since it is quasi-separated and $X^{\on{an}}\to Y^{\on{an}}$ is an \'etale cover whose source is a spatial diamond.
	    This shows that $Y$ is a kimberlite.
	    Moreover, since $X^\red\to Y^\red$ is a qcqs \'etale map, by \cite[Lemma 4.25]{Gle24} $X\to Y$ is also qcqs.
	    We let $\calF\to X$ be a qcqs formally adic v-cover where the $\calF$ is an affine formal v-sheaves.
	    Then $\calF\to Y$ is a qcqs formally adic v-cover with source a formal v-sheaf. 
	    Moreover, any map $\Spa C\to Y$ lifts to a map $\Spa C\to X$ and formalizes over $X$, consequently it formalizes over $Y$.
	    This finishes the proof that $Y$ is an affine spatial kimberlite.
    \end{proof}

    \subsection{The thick-reduced decomposition}
    One can verify from the definition that for every scheme $S\in \PSch$ the v-sheaf $S^\diamond$ is a spatial kimberlite, but it has empty analytic locus. 
    We can call this type of spatial kimberlite \textit{reduced}, they typically arise from the reduction functor adjunction.  
    The following is the complementary notion and the more relevant object to study in the theory of kimberlites.

    \begin{definition}
	    \label{thick kimber}
	    We say that a spatial kimberlite $X$ is \textit{thick} if the specialization map $\on{sp}:|X^{\on{an}}|\to |X^\red|$ is surjective.
    \end{definition}

    As we show below every spatial kimberlite can be decomposed on a thick spatial kimberlite and a reduced spatial kimberlite.
    We make some preparations, the following lemma is completely analogous to {\cite[Lemma 2.2.3]{vcesnavivcius2024purity}.

    \begin{lemma}
	    \label{productofpointsandvaluations}
	    Let $Y=\Spd B$ an affine formal v-sheaf. 
	    Then there are a collection of perfect valuation rings $\{V_j\}_{j\in J_1}$ endowed with the discrete topology, a collection of algebraically closed non-Archimedean fields $\{C_j\}_{j\in J_2}$ together with open bounded valuation subrings $\{C^+_j\subseteq C_j\}_{j\in J_2}$, and a qcqs formally adic v-cover $\Spd( \prod_{j\in J_2} C_j^+)\coprod \Spd(\prod_{j\in J_1}V_j)\to \Spd B$. 
	    Moreover, if the specialization map $\on{sp}:Y^{\on{an}}\to \Spec B_\red$ is surjective already $\Spd( \prod_{j\in J_2} C_j^+)\to \Spd B$ is a v-cover.
    \end{lemma}
    \begin{proof}
	    Let $I\subseteq B$ be the ideal that defines the topology on $B$ and let $\{i_1,\dots,i_n,p\}\subseteq I$ be a set of generators.
	    For every point $x\in \Spd B$ we can find a geometric point $\Spa(C_x,C_x^+)\to \Spd B$.
	    This gives a map of rings (without topology) $B\to \prod_x C_x^+$.
	    We claim that $\prod_x C_x^+$ is $I$-adically complete as a ring.
	    Indeed, $(\prod_x C_x^+)/I^n=\prod_x (C_x^+/I^n)$ since $I$ is finitely generated.
	    Moreover, products commute with limits and each $C^+_x$ is $I$-adically complete, the claim follows.

	    We endow $R^+:=\prod_x C_x^+$ with the $I$-adic topology, then by \cite[Lemma 2.23]{Gle24} the map $\Spd R^+\to \Spd B$ is a qcqs formally adic v-cover since by construction $|\Spd R^+|\to |\Spd B|$ is surjective.
	    We can separate the product in two factors $R^+=R_1^+\times R_2^+$ where $R_1^+$ corresponds to the factors where $I=0$ and $R_2^+$ corresponds to the factors where at least one of the generators $i_j$ is a non-zero divisor in $C_x^+$. 
	    Then $R_1^+=\prod_{j\in J_1} V_j$ and $R_2^+=\prod_{j\in J_2} C^+_j$. 
	    
	    Now, for the second statement observe that $(\Spd R^+)^{\on{an}}=(\Spd R_2^+)^{\on{an}}$ and that the map $\Spd R^+_2\to \Spd B$ is formally adic and qcqs.  
	    By \Cref{annoyinglemmaonsurjectivityofspecialization} the map $|\Spd R^+_2|\to |\Spd B|$ is surjective, which implies it is a v-cover. 
    \end{proof}

\begin{lemma}
	\label{rigproetalecover}
	If $X$ is a thick spatial kimberlite, then there is a universally open quasi-pro-\'etale map $f:\Spa (R,R^+)\to X^{\on{an}}$ such that $\Spa( R,R^+)$ is a strictly totally disconnected space and $f$ is formalizable. Moreover, the formalization $\Spd R^+\to X$ is a qcqs v-cover. 
\end{lemma}
\begin{proof}
We may assume $X$ is an affine spatial kimberlite.
Let $\Spd B\to X$ be a formally adic qcqs v-cover. 
By \Cref{productofpointsandvaluations}, we may assume that $\Spd B=\Spd R_1^+\coprod \Spd R_2^+$ with $\Spa( R_1,R_1^+)$ a product of points. 

Since $X^{\on{an}}$ is a spatial diamond, there is a universally open quasi-pro-\'etale v-cover $f:\Spa (R,R^+)\to X^{\on{an}}$ (see \cite[Proposition 11.24]{Sch17}). 
By \Cref{prekimbproetaleformalizing}, $f$ is formalizable. 
We claim that $\Spd R_1^+\to X$ factors through a map $\Spd R_1^+\to \Spd R^+$.
	    It suffices to prove that this happens on analytic loci by uniqueness of formalizations.
	    Consider $W=\Spa (R,R^+)\times_{X^{\on{an}}} \Spa (R_1,R_1^+)$, it suffices to find a section to the map $W\to \Spa (R_1,R_1^+)$, but this is a quasi-pro-\'etale cover and $\Spa (R_1,R_1^+)$ is extremally disconnected so it admits a section.

	    Now, $\Spd R_1^+\to \Spd B$ is formally adic and consequently qcqs. 
	    This gives that $\Spd R_1^+\to X$ is qcqs and since $|\Spd R_1^+|\to |X|$ is surjective by \Cref{annoyinglemmaonsurjectivityofspecialization} it is also a v-cover. 
	    This shows that $\Spd R^+\to X$ is also a v-cover.
\end{proof}

\begin{definition}
Let $X$ be a spatial kimberlite. We let $X^\thi\subseteq X$ denote the smallest subsheaf with the property that if $Y$ is a thick spatial kimberlite and $f:Y\to X$ is a formally adic map then $f$ factors through $X^\thi$.
\end{definition}

\begin{proposition}
	    \label{splittedanalyticvsnonanalyticcover}
Let $X$ be a spatial kimberlite. Then $X^\thi\to X$ is a spatial kimberlite and the map $X^\thi \to X$ is a formally adic closed immersion.
In particular, $X^\Red\coprod X^\thi\to X$ is a v-cover and $X$ is obtained from $X^\Red$ and $X^\thi$ by glueing along their intersection.
\end{proposition}
\begin{proof}
	Let $Z$ denote $\on{sp}(X^\an)\subseteq X^\red$.
	It follows from \Cref{rigproetalecover} that if $\Spa( R,R^+)\to X^\an$ is a formalizable v-cover, then the formalization $\Spd R^+\to X$ formally adic, qcqs and $X^\thi$ is its sheaf-theoretic image. 
Taking reduction, we get a diagram $\Spd R^+_\red \to X^\thi\times_X X^\Red\to X^\Red$, and the sheaf theoretic image of $\Spd R^+_\red\to X^\Red$ is $Z^\diamond$, so $X^\thi\times_X X^\Red=Z^\diamond$. 
This implies that $X^\thi$ is a prekimberlite, that $X^\thi\to X$ is qcqs, formally adic and that $(X^\thi)^\red=Z$.
To prove $X^\thi\to X$ is closed we may apply the valuative criterion \cite[Proposition 18.3]{Sch17}.
But $|(X^\thi)|=|X^\an|\cup |Z^\diamond|$ and the maps $X^\an\to X$ and $Z^\diamond\to X$ are partially proper.

By \Cref{formallyadicclosedisvaluative}, $X$ is a valuative prekimberlite and since $(X^\thi)^\an=X^\an$ it is also a kimberlite.
Since $X$ formalizes geometric points and $X^\thi\subseteq X$ is formally adic and closed $X^\thi$ also does. 
Since $\Spd R^+\to X^\thi$ is a qcqs formally adic surjection, $X^\thi$ is spatial.
This finishes showing the first claim.
 
Now, the map $X^\Red\coprod X^\thi\to X$ is qcqs and surjective, so it is a v-cover. 
In any topos, if $Y\to X$ is a surjective map, then 

\[
X= \on{co.eq}(R \;\substack{\xrightarrow{f}\\[-0.6ex]\xrightarrow[g]{} }\; Y) \]
where $R=Y\times_X Y$.
Since $X^\thi$ and $X^\Red$ are closed immersions, $X^\thi\cap X^\Red$ is also a closed immersion, and $R$ precisely defines data glueing $X^\thi$ to $X^\Red$ along $X^\thi\cap X^\Red$.
\end{proof}


\subsection{Being spatial is stable under formal neighborhoods}

The following statement was our main motivation to define spatial kimberlites.

    \begin{proposition}
	    \label{keytubularneighborhood}
	    Let $X$ be an affine spatial kimberlite and $Z\subseteq X^\red$ a constructible Zariski closed subset, then $\widehat{X}_{/Z}$ is an affine spatial kimberlite.
    \end{proposition}
    \begin{proof}
	    Let $S=X^\Red$.
	    By \cite[Proposition 4.21, 4.34]{Gle24}, $\widehat{X}_{/Z}$ is an affine valuative prekimberlite, with $(\widehat{X}_{/Z})^\red=Z$.
	    We now prove $(\widehat{X}_{/Z})^\an$ is a spatial v-sheaf (see \cite[Definition 12.12]{Sch17}). 
	    By \Cref{quasiseparatednessofprekimberlites}, it is quasiseparated. 
	    Consider the cover $S \coprod X^\thi\to X$ of \Cref{splittedanalyticvsnonanalyticcover}, using \Cref{rigproetalecover} we can refine it to a v-cover $W=S\coprod \Spd R^+ \to X$, more precisely we get a qcqs formally adic v-cover $f:W\to X$ such that $W^\an\to X^\an$ is universally open.
	    Let $Z_W:=f^{-1}(Z)\subseteq W^\red$ and let $S_W=W^\Red\subseteq W$.
	    By \cite[Proposition 4.20]{Gle24}, we get a qcqs formally adic v-cover $\widehat{W}_{/Z_W}\to \widehat{X}_{/Z}$. 
	    By \cite[Proposition 4.19]{Gle24}, $\widehat{W}_{/Z_W}^\an$ is a spatial diamond. 
	    Since $(\widehat{W}_{/Z_W})^\an$ is qcqs and surjective over $(\widehat{X}_{/Z})^\an$, then $(\widehat{X}_{/Z})^\an$ is also qcqs.
	    We now prove every point of $(\widehat{X}_{/Z})^\an$ has a basis of open neighborhoods that are qc over $(\widehat{X}_{/Z})^\an$. 
	    This is clear for $x\in X^\an \cap (\widehat{X}_{/Z})$ since this is an open subset of a spatial diamond (see \cite[Proposition 4.22]{Gle24}). 

	    Let $x\in S \cap (\widehat{X}_{/Z})^\an$, and let $U\subseteq (\widehat{X}_{/Z})^\an$ an open subset containing $x$.
	    Since $S \cap (\widehat{X}_{/Z})=\widehat{S}_{/Z}$ is a kimberlite (even a formal v-sheaf), there is an open subset $V\subseteq U$ containing $x$ such that $V\cap S$ is qc over $(\widehat{S}_{/Z})^\an$.
	    Let $V_W=f^{-1}(V)\subseteq \widehat{W}_{/Z}^\an$. 
	    Observe that $V_W\cap S_W$ is qc over $\widehat{S_W}_{/Z_W}$.
	    We may find a qc open subset $V'_W\subseteq V_W\subseteq \widehat{W}_{/Z}^\an$ such that $V'_W\cap S_W=V_W\cap S_W$.
	    Indeed, since $\widehat{W}_{/Z}^\an$ is a spatial diamond and $V_W$ is an open subset we may cover $V_W$ by a union of open qc open subsets. 
	    From this, a priori infinite, family we may find finitely many opens whose restriction to $S_W$ covers $V_W\cap S_W$.
	    Let $V'\subseteq V$ denote the sheaf theoretic image of $V'_W\to V$, we claim $V'$ is qc open subsheaf containing $x$ and contained in $U$.
	    Containment in $U$ follows from 
	    \[|V'|=f(|V'_W|)\subseteq f(|V_W|)=|V|\subseteq |U|,\]
	    so it suffices to prove that the image of $|V'_W|\to |V|$ is open. 
	    Indeed, $V'_W\to (\widehat{X}_{/Z})^\an$ is the composition of quasicompact maps, so it surjects onto its topological image.
	    Since $|V|$ has the quotient topology along the map $|V_W|\to |V|$, it suffices to prove that $f^{-1}(f(|V'_W|))\subseteq |V_W|$ is open. 

	    Observe that by construction, $S_W\cap V'_W=S_W\cap f^{-1}(f(|V'_W|))=S_W\cap V_W$.
	    In particular, if $x\in S_W\cap f^{-1}(f(|V'_W|))$ there is an open subset containing $x$ and contained in $f^{-1}(f(|V'_W|))$, namely $V'_W$.
	    On the other hand, $W^\an \cap f^{-1}(f(|V'_W|))=f^{-1}(f(W^\an\cap V'_W))$, and since the map $f:W^\an\to X^\an$ is universally open, $f(V'_W\cap W^\an)\subseteq V\cap X^\an$ is an open subset. 
	    Consequently, $W^\an \cap f^{-1}(f(|V'_W|))$ is also open.
	    This finishes showing that $|V'|\subseteq |V|$ is open.

	    Now, since $V'$ is covered by $V'_W$ it is qc, and gives rise to a qc neighborhood of $x$ contained in $U$.
	    Since $U$ was arbitrary, we have shown that there is a basis of qcqs open neighborhoods of $x$. This finishes showing that $(\widehat{X}_{/Z})^\an$ is a spatial v-sheaf.  
	    
	    By \cite[Theorem 12.18]{Sch17}, to prove that $\widehat{X}_{/Z}^\an$ is a spatial diamond it suffices to find for every $x\in \widehat{X}_{/Z}^\an$ a quasi-pro-\'etale map $\iota_x:\Spa(C_x,C_x^+)\to \widehat{X}_{/Z}^\an$ mapping to $x$.
	    Over $X^\an\cap \widehat{X}_{/Z}^\an$ this is possible since $X^\an$ is by hypothesis a spatial diamond.
	    On the other hand, over $S\cap \widehat{X}_{/Z}^\an$ this is possible since $\widehat{S}_{/Z}^\an$ is a spatial diamond and the map $\widehat{S}_{/Z}^\an\to \widehat{X}_{/Z}^\an$ is a closed immersion, hence quasi-pro-\'etale.
	    This finishes the proof that $\widehat{X}_{/Z}$ is an affine kimberlite. 
	    Now, it is an affine spatial kimberlite since the map $\widehat{W}_{/Z}\to \widehat{X}_{/Z}$ is an adic qcqs v-cover and $\widehat{W}_{/Z}$ is an affine formal v-sheaf. 
	    Moreover, by hypothesis, geometric points of $\widehat{X}_{/Z}$ formalize in $X$ and by \cite[Proposition 4.20]{Gle24} any such formalization factors through $\widehat{X}_{/Z}$ so that $\widehat{X}_{/Z}$ also formalizes geometric points.
    \end{proof}

    \subsection{Spatial kimberlites vs. spatial diamonds}

    \begin{proposition}
	    \label{representabilityinlocallyspatialkimberlites}
    Let $X$ be an affine kimberlite that formalizes geometric points, then $X$ is spatial if and only if $X\to \ast$ is representable in locally spatial diamonds. 
    \end{proposition}
    \begin{proof}
	    We prove the first claim.
	   Suppose that $X$ is an affine spatial kimberlite.
	    Let $U=\Spa(R,R^+)$ be a strictly totally disconnected perfectoid space, by \cite[Proposition 11.24, Remark 11.25]{Sch17} it suffices to show that $X\times U$ is a locally spatial diamond. 	
	    Let $S=(\Spec R^+)^\diamond$ and consider $W:=X\times S$.
	    We claim that $W$ is an affine spatial kimberlite. 
	    By \Cref{valuativeprekimbfiberprod}, $W$ is an affine valuative prekimberlite.
	    Observe that the projection map $\pi_1:W\to X$ is formally adic so that $W^\an=\pi^{-1}(X^\an)=X^\an\times S$.
	    Since $S\to \ast$ is representable in spatial diamonds, $W^\an$ is again a spatial diamond.
	    This shows that $W$ is a kimberlite.

	    Let $\Spd B\to X$ be a qcqs formally adic v-cover with $B$ a ring in characteristic $p$ and $I\subseteq B$ an ideal of definition. 
	    Then $\Spd B\times S\to X\times S$ is a qcqs formally adic v-cover, and $\Spd B\times S$ is an affine formal v-sheaf corresponding to the $I$-adic completion of $R^+\otimes_{\bbF_p} B$. 
	    Now, $W$ formalizes geometric points since $X$ and $S$ do, this finishes proving that $W$ is an affine spatial kimberlite.

	    Let $\varpi\in R^+$ be a pseudo-uniformizer. 
	    This defines a constructible Zariski closed subset $Z\subseteq \Spec R^+$ of $S^\red$ and $Z_W\subseteq X^\red\times S^\red=X^\red\times \Spec R^+$ of $W^\red$.  
	    By \Cref{keytubularneighborhood}, $\widehat{W}_{Z_W}$ is an affine spatial kimberlite, and in particular $\widehat{W}_{Z_W}^\an$ is a spatial diamond. 
	   But $X\times U$ is the open locus in $\widehat{W}_{Z_W}^\an$ where $\varpi\neq 0$, and any open subsheaf of a spatial diamond is a locally spatial diamond.
	   
	   Conversely, assume that $X$ is an affine kimberlite that formalizes geometric points, and that $X\to \ast$ is representable in locally spatial diamonds. 
	   Let $\Spa (R,R^+)\to X^{\on{an}}$ be a universally open quasi-pro\'etale cover by a strictly totally disconnected space, and let $W=\Spd R^+\coprod X^\Red$.
	   We have a map $f:W\to X$, by \Cref{prekimbproetaleformalizing}, we claim that $W\to X$ is a formally adic qcqs v-cover. 
	   Now, $(X\times \Spd \bbF_p\pot{t})^{\on{an}}$ is a spatial diamond. 
	   Indeed, $X^{\on{an}}\times \Spd \bbF_p\pot{t}$ is a locally spatial diamond since $\Spd \bbF_p\pot{t}\to \ast$ is representable in locally spatial diamonds and $X^\an$ is by hypothesis a spatial diamond, and similarly $X\times \Spd \bbF_p\rpot{t}$ is also a locally spatial diamonds. 
	   Moreover, these two open subsheaves cover $(X\times \Spd \bbF_p\pot{t})^{\on{an}}$.
	   Moreover, $(X\times \Spd \bbF_p\pot{t})^\an$ is quasiseparated by \Cref{quasiseparatednessofprekimberlites}. 
	   Consequently, the map $f:|(W\times \Spd \bbF_p\pot{t})^{\on{an}}|\to | (X\times \Spd \bbF_p\pot{t})^{\on{an}}|$ is a surjective and spectral map of locally spectral spaces, each of which is quasiseparated.
	   Further, $|(W\times \Spd \bbF_p\pot{t})^{\on{an}}|$ is quasicompact, so $|(X\times \Spd \bbF_p\pot{t})^{\on{an}}|$ is also quasicompact. 
	   It follows from \cite[Proposition 11.19.(iii)]{Sch17}, that $(X\times \Spd \bbF_p\pot{t})^{\on{an}}$ is a spatial diamond.

	   The argument also shows that $f$ is a qcqs v-cover.
	   Indeed, we have a Cartesian diagram 
	   \begin{center}
	   \begin{tikzcd}
		   (W\times \Spd \bbF_p\pot{t})^{\on{an}}  \arrow{r}{g} \arrow{d}  & (X\times \Spd \bbF_p\pot{t})^{\on{an}} \arrow{d} \\
		   W \arrow{r}{f} & X,
	   \end{tikzcd}
	   \end{center}
	   and $g$ is a qcqs v-cover since it is a map of spatial diamonds that surjective on topological spaces.
	   Since $X\times \Spd \bbF_p\rpot{t}\to X$ and factors through $(X\times \Spd \bbF_p\pot{t})^{\on{an}}$, it follows from \cite[Proposition 10.11]{Sch17} that the map $W\to X$ is also qcqs and a v-cover.
    \end{proof}

    \begin{corollary}
    	Any map of locally spatial kimberlites, $X\to Y$, is representable in locally spatial diamonds.
    \end{corollary}
    \begin{proof}
	   The claim is local on the source and target so we may assume $X$ and $Y$ are affine spatial kimberlites. 
	   Since the maps $X,Y\to \ast$ are representable in locally spatial diamonds, the claim easily follows from \cite[Corollary 11.29]{Sch17}.
    \end{proof}

    \begin{proposition}
	    \label{closedimmersionpreservespatial}
	    Suppose that $f:Z\to X$ is a formally adic closed immersion of v-sheaves and that $X$ is an affine spatial kimberlite, then $Z$ is an affine spatial kimberlite. 
    \end{proposition}
    \begin{proof}
	    By \cite[Proposition 4.41.(4)]{Gle24}, $Z$ is an affine kimberlite.
	    Let us show that $Z$ formalizes geometric points.
	    Let $\Spa (C,C^+)\to Z$ be a geometric point, since $X$ is spatial the induced geometric point formalizes to a map $\Spd C^+\to X$.
	    The base change $\Spd C^+\times_X Z$ is a formally closed subsheaf of $\Spd C^+$ containing $\Spa (C,C^+)$, so it must agree with $\Spd C^+$ (see \Cref{fixed up version formally dense}).
	    By \Cref{representabilityinlocallyspatialkimberlites}, it suffices to prove that $Z\to \ast$ is representable in locally spatial diamonds, but $Z\to X$ is (see \cite[Proposition 11.20]{Sch17}) and $X\to \ast$ also is. 
	    This shows that $Z\to \ast$ is representable in locally spatial diamonds.
%
    \end{proof}

%

    \begin{proposition}
	    \label{stable-under-cartesian}
	    The category of spatial kimberlites (respectively locally spatial kimberlites, respectively affine spatial kimberlites), is stable under fiber products and contains $\ast$.	
    \end{proposition}
    \begin{proof}
	    We only show the case of affine spatial kimberlites since the other claims easily follow from this one.
	    Suppose that $X$, $Y$ and $Z$ are affine spatial kimberlites, and $W=X\times_Y Z$.
	    We may write $W$ as the base change of the map $X\times Z\to Y\times Y$ by the diagonal $Y\to Y\times Y$.
	    By \Cref{closedimmersionpreservespatial}, it suffices to prove that $X\times Z$ is a spatial kimberlite. 
	    Let $\Spd B_X\to X$ and $\Spd B_Z\to Z$ be qcqs formally adic v-covers.

	    By \Cref{valuativeprekimbfiberprod}, $X\times Z$ is a valuative prekimberlite and 
	    \[(X\times Z)^{\on{an}}=(X^{\on{an}}\times Z)\cup (X\times Z^{\on{an}}).\]
	    Also, by \Cref{representabilityinlocallyspatialkimberlites}, $X\to \ast$ and $Z\to \ast$ are representable in locally spatial diamonds and since $X^{\on{an}}$ and $Z^{\on{an}}$ are spatial diamonds, then $(X\times Z)^{\on{an}}$ is a qs locally spatial diamond. 
	    Moreover, since $(\Spd B_X\times \Spd B_Z)^{\on{an}}\to (X\times Z)^{\on{an}}$ is a v-cover and the source is qc we conclude that $(X\times Z)^{\on{an}}$ is a spatial diamond and that $X\times Z$ is a kimberlite. 
	    Furthermore, the map $(\Spd B_X\times \Spd B_Z)\to (X\times Z)$ is a qcqs formally adic v-cover, and $\Spd B_X\times \Spd B_Z$ is an affine formal v-sheaf.
	    Finally, it formalizes geometric points since $X$ and $Z$ both do.
    \end{proof}

    \subsection{Quasicompact = formally adic}

    \begin{proposition}
	    \label{qcqs-is-formal-adic}
    Let $f:X\to Y$ be a map of spatial kimberlites. Then $f$ is formally adic if and only if it is quasicompact.
    \end{proposition}
    \begin{proof}
	    The properties of being formally adic and of being quasicompact are local on the source and target as long as we use qcqs formally adic v-covers.
	    In particular, we may assume that $X$ and $Y$ are affine formal v-sheaves (i.e. $X=\Spd B$ and that $Y=\Spd A$).
	    By \cite[Lemma 2.23]{Gle24}, if the map $X\to Y$ is formally adic, then it is quasicompact.

	    Suppose that $X\to Y$ is quasicompact and let $Z=X\times_Y Y^\Red$, this is an affine spatial kimberlite by \Cref{stable-under-cartesian}. 
	    Suppose there is $z\in Z^{\on{an}}$ of rank $1$. 
	    The subsheaf of points that factor through $z$ is a spatial diamond since $X^\an$ is a spatial diamond, and it can be covered by a geometric point $\Spa(C,O_C)\to z$.
	    We can formalize this to a map $\Spd O_C\to Z$ which is quasicompact since it is formally adic. 
	    By assumption, the composition $\Spd O_C\to Z\to  Y^\Red\to \ast$ is again quasicompact.
	    But $\Spd O_C\to \ast$ is not quasicompact. 
	    The contradiction shows that $Z^{\on{an}}=\emptyset$, so $f$ is formally adic. 
    \end{proof}

    We will need the following lemma when we discuss examples of spatial kimberlites.
    \begin{lemma}
	    \label{kimberlite-not-so-different-from-spatial-kimb}
Let $X$ be a spatial kimberlite, let $Y$ be a kimberlite and let $f:X\to Y$ be a formally adic map. The following hold:
	    \begin{enumerate}
		    \item If $f$ is a v-cover, then it is qcqs. 
		    \item If $|f|$ is surjective and $Y\times \Spd \bbF_p\pot{t}$ is a kimberlite, then $f$ is a formally adic cover. 
	    \end{enumerate}
    \end{lemma}
    \begin{proof}
    Observe that $f$ is automatically quasiseparated since both v-sheaves are separated over $\ast$.
    Assume first that $f$ is v-surjective.
    Now, quasi compactness can be checked v-locally, so it suffices to prove that $X\times_Y X\to X$ is quasicompact.
    Since $X\times_Y X$ is the base change of $X\times X\to Y\times Y$ along the diagonal map $\Delta_Y$, by \Cref{closedimmersionpreservespatial} and \Cref{stable-under-cartesian}, we can conclude that $X\times_Y X$ is a spatial kimberlite.
    Moreover, the map $X\times X\to Y\times Y$ is formally adic, since $X\to Y$ is. 
    This implies that $X\times_Y X\to Y$ is formally adic, and consequently $X\times_Y X\to X$ is also formally adic. 
    We can now apply \Cref{qcqs-is-formal-adic} to show that $X\times_Y X\to X$ is quasicompact.
    This finishes the proof of the first claim.
    
    For the second claim, we consider the map $X\times \Spd \bbF_p\pot{t}\to Y\times \Spd \bbF_p\pot{t}$.
    By hypothesis, it is formally adic and surjective at the level of topological spaces. 
    We get a map $(X\times \Spd \bbF_p\pot{t})^{\on{an}}\to (Y\times \Spd \bbF_p\pot{t})^{\on{an}}$ of spatial diamonds that is surjective at the level of topological spaces, so it is v-surjective. Now, it follows that the map $X\times \Spd \bbF_p\rpot{t}\to Y$ is v-surjective and since it factors through $f$, then $f$ is also v-surjective. 
    \end{proof}
    \section{Examples of spatial kimberlites} 
    For this section we let $E$ be a local field with $O_E\subseteq E$ its ring of integers, $\pi\in O_E$ a uniformizer and $k=\bbF_q$ be its residue field.
    We let $\breve{E}$ denote the $\pi$-adic completion of the maximal unramified extension of $E$. 
    We fix $\calG$ a parahoric group scheme over $\on{Spec} O_E$ with reductive generic fiber $G$ over $E$.
    We also fix $\mu:\bbG_m\to G_{\overline{E}}$ a conjugacy class of geometric cocharacters with reflex field $F$.
    For any perfect ring $A$ in characteristic $p$ we let $\bbW(A)$ denote the ring of $O_E$-Witt vectors.
    More precisely, if $E$ is of characteristic $0$, then $\bbW A\coloneqq\calW(A)\otimes_{\bbZ_p}O_E$, where $\calW(A)$ denotes the $p$-typical Witt vectors, and if $E$ is of characteristic $p$ then $\bbW A=A\hat{\otimes}_{\bbF_q} O_E\cong A[[\pi]]$.

    If $S=\Spa(R,R^+)$, we let $\bbA_{\on{inf}}(R^+)$ denote $\bbW(R^+)$ endowed with the $(\pi,[\varpi])$-adic topology, where $\varpi$ denotes a(ny) pseudo-uniformizer of $R^+$.
    We let $\calY_S$ be the locus in $\Spa \bbA_{\on{inf}}(R^+)$ where $[\varpi]\neq 0$ for some pseudo-uniformizer $\varpi\in R^+$, and $Y_S$ be the locus in $\Spa \bbA_{\on{inf}}(R^+)$ where $\pi\cdot [\varpi]\neq 0$. 
    These are sous-perfectoid adic spaces \cite[Proposition II.1.1]{FS21}.
    \subsection{Positive absolute Banach Colmez spaces}
    Recall that given a $\lambda> 0$ one can consider the absolute Banach--Colmez space $\calB\calC(\calO(\lambda))$ as a v-sheaf \cite[$\mathsection$ II.2.2]{FS21}.
    Moreover, $\calB\calC(\calO(\lambda))$ has a unique non-analytic point $0:\ast \to\calB\calC(\calO(\lambda))$ and $\calB\calC(\calO(\lambda))\setminus \{0\}$ is a spatial diamond. 
    We will show that the v-sheaves $\calB\calC(\calO(\lambda))\setminus \{0\}$ are spatial kimberlites.
    \begin{theorem}
	    \label{BCareSpatialKbaby}
	    Let $D$ be an isocrystal with $\calE_D=\oplus_{\lambda\in \bbQ_{\geq 0}} \calO(\lambda)^{m_\lambda}$, the following hold:
	    \begin{enumerate}
		    \item $\calB\calC(D)$ is a locally spatial kimberlite with 
			    \[\calB\calC(D)^\Red\cong \calB\calC(\calO^{m_0})=\underline{E}^{m_0}.\]
		    \item If $m_0=0$ then $\calB\calC(D)$ is a pointed spatial kimberlite.
	    	
	    \end{enumerate}
    \end{theorem}
    \begin{remark}
    It is natural to ask if negative Banach--Colmez spaces are also spatial kimberlites. 
    The answer is \textbf{no} since these spaces are not v-locally formal. Indeed, their unique non-analytic point has a meromorphic limit rather than a formal limit.	
    \end{remark}
    \begin{proof}
	    Since $\calB\calC(\oplus_{\lambda\in \bbQ_{\geq 0}} \calO(\lambda)^{m_\lambda})= \calB\calC(\calO^{m_0})\times \prod_{\lambda\in \bbQ_{> 0 } }\calB\calC(\calO(\lambda))$ we can deduce the first statement from the second and from \cite[Proposition II.2.5.(ii)]{FS21}.
	    Moreover, by \Cref{stable-under-cartesian} we may reduce to the case where $m_\lambda=1$ for a chosen $\lambda$ and $m_\lambda=0$ otherwise. 
	    From now on $D$ has the form $D=\calO(\lambda)$ for some $\lambda>0$ and with $\lambda=\frac{n}{m}$ with $n$ and $m$ relatively prime. 
	    Let us prove $\calB\calC(D)$ is formally separated. 
	    The argument for separatedness can be found in \cite[Proposition II.2.16]{FS21}.
	    By \cite[Corollary 3.20]{GI_Bunmer}, $\calB\calC(\calO(\lambda))^{\on{red}}=\ast$ and the adjunction map is the inclusion of the origin. 
	    The diagonal $\Delta:\calB\calC(\calO(\lambda))\to \calB\calC(\calO(\lambda))^2$ is formally adic since only the origin maps to the origin. 
	    This finishes showing that $\calB\calC(\calO(\lambda))$ is formally separated. 

	    To prove that $\calB\calC(D)$ is v-locally formal we make some preparations.
    Recall that for any affinoid perfectoid $S=\Spa(R,R^+)$ and a choice of pseudo-uniformizer $\varpi\in R^+$, we can consider the ring $B^{S}_{[1,\infty]}=\Gamma(Y_{S,[1,\infty]},\calO)$. 
    We let $R^+_{\on{red}}=(R^+/\varpi)^{\on{perf}}$ and $R^\circ_{\on{red}}=(R^\circ/\varpi)^{\on{perf}}$.
    \begin{lemma}
	    \label{ring-computation}
	    Let $S=\Spa(R,R^+)$. The following diagram of rings is Cartesian:	
    \begin{center}
    \begin{tikzcd}
	    B^S_{[1,\infty]} \arrow{r} \arrow{d}  & \bbW(R^+_{\on{red}})[\frac{1}{\pi}] \arrow{d} \\
     B^{S}_{[1,\infty]} \arrow{r} & \bbW(R^\circ_{\on{red}})[\frac{1}{\pi}]
    \end{tikzcd}
    \end{center}
    \end{lemma}
    \begin{proof}
	    We do this in steps. Observe that the following diagram is Cartesian,
    \begin{center}
    \begin{tikzcd}
	    \bbW(R^+)[T]/(\pi\cdot T- [\varpi],\pi^n) \arrow{r} \arrow{d}  & \bbW(R^+_{\on{red}})[T]/(\pi\cdot T,\pi^n) \arrow{d} \\
	    \bbW(R^\circ)[T]/(\pi\cdot T- [\varpi],\pi^n) \arrow{r} & \bbW(R^\circ_{\on{red}})[T]/(\pi\cdot T,\pi^n).
    \end{tikzcd}
    \end{center}
    Indeed, in general whenever  
    \begin{center}
    \begin{tikzcd}
     A \arrow{r} \arrow{d}  & B \arrow{d} \\
     C \arrow{r} & D
    \end{tikzcd}
    \end{center}
    is a Cartesian diagram of rings with $C\to D$ surjective and $I\subseteq A$ is an ideal such that $A/I\to C/I$ is injective, then 
    \begin{center}
    \begin{tikzcd}
	    A/I \arrow{r} \arrow{d}  & B/I \arrow{d} \\
     C/I \arrow{r} & D/I
    \end{tikzcd}
    \end{center}
    is also Cartesian.

    To wit, take $[c]\in C/I$ and $[b]\in B/I$ mapping to the same element in $D/I$. We have that $c-(a_1d_1+\dots a_nd_n)=b$ in $D$ and we can lift each $d_i$ to elements $c_i$ so that replacing $c$ by $c-(a_1c_1+\dots a_nc_n)$ we can arrange that $c=b$ in $D$ without changing the class $[c]\in C/I$. In particular, we can lift to an element $a\in A$ inducing an element $[a]\in A/I$ mapping to $[b]$ and to $[c]$. This proves that $A/I\to C/I\times_{D/I} B/I$ is surjective. Since by assumption $A/I\to C/I$ is injective we can say the same about the map to $A/I\to C/I\times_{D/I} B/I$ so that it is an isomorphism.

    Now, passing to limits as $n\to \infty$ and using that limits commute with Cartesian diagrams we get the following Cartesian diagram:
    \begin{center}
    \begin{tikzcd}
	    \bbW(R^+)\langle T\rangle/(\pi\cdot T- [\varpi]) \arrow{r} \arrow{d}  & \bbW(R^+_{\on{red}})\langle T\rangle /(\pi\cdot T) \arrow{d} \\
	    \bbW(R^\circ)\langle T\rangle/(\pi\cdot T- [\varpi]) \arrow{r} & \bbW(R^\circ_{\on{red}})\langle T\rangle /(\pi\cdot T)
    \end{tikzcd}
    \end{center}
    The diagram in the statement is constructed from the one above by inverting $\pi$, and it remains Cartesian. Indeed, we can realize the localization by $\pi$ as a filtered colimit along transition morphisms given by multiplication by $\pi$. Since filtered colimits commute with finite limits the remaining diagram is still Cartesian.
    \end{proof}

    We now prove that $\calB\calC(D)$ formalizes every map coming from an affinoid perfectoid $(R,R^+)$. In particular, it is v-formalizing.
    Let $S=\Spa(R,R^+)$ and let $S_0=\Spa(R,R^\circ)$.
    Recall that restriction along the map $Y_{S_0,[1,\infty)}\to Y_{S_0,[1,\infty]}$ induces a fully-faithful functor of categories of $\varphi$-equivariant vector bundles (see \cite[Proposition 2.1.4]{PR21}).
    Using \cite[Theorem 3.17]{GI_Bunmer} we can reinterpret the full-faithfulness result we just mentioned to prove that the inclusion map $\Spa(R,R^\circ)\to \Spd(R^\circ, R^\circ)$ induces a group isomorphism
    \[\calB\calC(D)(\Spd(R^\circ))\to \calB\calC(D)(\Spa(R,R^\circ)).\]
    Indeed, $Y_{\Spd(R^\circ)}=Y_{S_0,(0,\infty]}$ as in \cite[Remark 3.14]{GI_Bunmer}.
    Since the slope is positive, the Banach--Colmez space is defined as a space of global sections.
    An element $\alpha \in \calB\calC(D)(\Spd(R^\circ,R^\circ))$ can be interpreted as an element in $\alpha\in (B^{S_0}_{[1,\infty]})^m$ which is appropriately $\varphi$-equivariant and where $m=\on{rank}(\calE_D)$.
    Furthermore, this induces a map $\Spd(R^\circ_{\on{red}},R^\circ_{\on{red}})\to \calB\calC(D)$ which is necessarily the $0$ map, this shows that $\alpha=0$ in $\bbW(R^\circ_{\on{red}})^m$.
    By \Cref{ring-computation}, $\alpha \in  B^{S}_{[1,\infty]}$ which produces a map $\Spd(R^+,R^+)\to  \calB\calC(D)$ appealing again to \cite[Theorem 3.17]{GI_Bunmer} and \cite[Remark 3.14]{GI_Bunmer}.

    As of now, we have proved that $\calB\calC(\calO(\lambda))$ is a pointed prekimberlite that formalizes points. 
    By \cite[Proposition II.2.16]{FS21}, it is a valuative prekimberlite. 
    Indeed, the Heuer specialization map gets identified with the structure morphism $\calB\calC(D)\to \ast$ which is partially proper. 
    By \cite[Proposition II.3.7.(i)]{FS21} $\calB\calC(D)^{\on{an}}$ is a spatial diamond, which proves that $\calB\calC(D)$ is a kimberlite.
    By \Cref{representabilityinlocallyspatialkimberlites}, $\calB\calC(D)$ is spatial whenever $\calB\calC(D)\to \ast$ is representable in locally spatial diamonds. This is follows from \cite[Proposition II.2.16]{FS21}. 
    \end{proof}

    Recall that $G$ denotes a reductive group over $E$.
    Let $B(G)$ denote Kottwitz' set and let $b\in B(G)$. 
    Recall that for every $S=\Spa R\in \Perf_{\overline{\bbF}_p}$ the element $b$ defines a $G$-bundle $\calE_b$ over the relative Fargues--Fontaine curve $X_{S,\on{FF}}$. 
    This gives rise to a v-sheaf over $\Spd \overline{\bbF}_p$ parametrizing the automorphisms of $\calE_b$. It is denoted by $\widetilde{G}_b$.
    Moreover, we can write $\widetilde{G}_b=\widetilde{G}^{>0}_b\rtimes \underline{G_b(E)}$, as in \cite[Proposition III.5.1]{FS21}.
    \begin{theorem}
	    Let the notation be as above, the following hold.
	    \begin{enumerate}
		    \item $\widetilde{G}_b$ is a locally spatial kimberlite.
		    \item $\widetilde{G}^{>0}_b$ is a pointed spatial kimberlite. 
	    \end{enumerate}
    \end{theorem}
    \begin{proof}
    Let $\rho:G\to \on{GL}_n$ be a fully faithful embedding. We claim that $\widetilde{G}_b\to  \widetilde{G}_{\rho(b)}$ is a formally adic closed immersion.
    Let $S=\Spa (R,R^+)\in \Perf_{\overline{\bbF}_p}$. 
    We have a closed immersion of reductive group schemes $\on{Aut}_{X_{S, \on{FF}}}(\calE_b)\to \on{Aut}_{X_{S, \on{FF}}}(\calE_{\rho(b)})$ over ${X_{S, \on{FF}}}$. 
    Furthermore, $\on{Aut}_{X_{S, \on{FF}}}(\calE_{\rho(b)})$ is a Zariski closed subset of $\bbV(\calE nd^2)$, where $\bbV(\calE nd^2)$ denotes the total space of the endomorphism vector bundle $\calE nd^2:=\underline{\on{End}}_{X_{S, \on{FF}}}(\calE_{\rho(b)})^2$. 
    Let $\calI_{S}\subseteq \calO_{\bbV(\calE nd^2)}$ be the ideal sheaf defining $\on{Aut}_{X_{S, \on{FF}}}(\calE_b)$ as a closed subscheme of $\bbV(\calE nd^2)$.

    Any $\alpha\in \widetilde{G}_{\rho(b)}(R,R^+)$ corresponds to a section 
    \[\Gamma_\alpha:{X_{S, \on{FF}}} \to \on{Aut}_{X_{S, \on{FF}}}(\calE_{\rho(b)}),\]
    and $\Gamma_\alpha^{-1}(\calI_S)$ defines a Zariski closed subset of ${X_{S, \on{FF}}}$. 
    By \cite[Lemma IV.4.23]{FS21}, this defines a Zariski closed subset in $\Spa(R,R^+)$ representing the subsheaf of points whose composition with $\alpha$ factor through $\on{Aut}_{X_{S, \on{FF}}}(\calE_{b})$. 
    This proves that $\widetilde{G}_b\to \calB\calC(\underline{\on{End}}_{X_{S, \on{FF}}}(\calE_{\rho(b)})^2)$ is a closed immersion. 

    We can identify the map $\underline{{G}_b(E)}\to \widetilde{G}_b$ with the reduction adjunction $\widetilde{G}_b^\Red\to\widetilde{G}_b$ by \cite[Theorem 7.14]{GI_Bunmer}. 
    On the other hand $\underline{G_b(E)}$, is the moduli space of graded automorphisms of $\calE_b$ (see \cite[Proposition III.4.7]{FS21}). 
    Since $\rho:G\to \on{GL}_n$ is fully faithful, being graded can be verified in $\on{Aut}(\calE_{\rho(b)})$. 
    But $\calB\calC(\underline{\on{End}}(\calE_{\rho(b)}))^\Red$ also parametrizes graded endomorphisms. 
    This proves that $\widetilde{G}_b\to \calB\calC(\underline{\on{End}}(\calE_{\rho(b)})^2)$ is formally adic. 
    By \Cref{closedimmersionpreservespatial} and \Cref{BCareSpatialKbaby}, $\widetilde{G}_b$ is a locally spatial kimberlite. 
    On the other hand, $\widetilde{G}_b=\widetilde{G}^{>0}_b\times \underline{{G}_b(E)}$ sheaf theoretically (without group structure), and $\widetilde{G}_b^{\on{red}}=\underline{{G}_b(E)}$ which implies that $(\widetilde{G}^{>0}_b)^{\on{red}}=\Spec \overline{\bbF}_p$, so it is a pointed spatial kimberlite, by \Cref{stable-under-cartesian}. 
    \end{proof}
    \subsection{Beilinson--Drinfled affine Grassmanians}

    In our work \cite{AGLR22} with Ansch\"utz, Louren\c{c}o and Richarz the notion of flat $\pi$-adic kimberlite was introduced for v-sheaves with a structure map $X\to \Spd O$ and $O$ the ring of integers of a local field \cite[Definition 2.30]{AGLR22}. 
    This notion is a precursor of the notion of \textit{thick} spatial kimberlites used in this article (see \Cref{thick kimber}).
    For the convenience of the reader, we spell down \Cref{flatvsthick} below to facilitate comparability with the statements written in \cite{AGLR22}.

	    \begin{remark}
	    At the time that \cite{AGLR22} was written, the word kimberlite referred to a weakened version of what we call kimberlites in this article and in the latest version of \cite{Gle24}.
	    Indeed, when \cite{AGLR22} was written, the Heuer specialization map for specializing v-sheaves had not been introduced and we had not found the definition of valuative prekimberlite which is the correct thing to consider when one does not require properness. 
	    Moreover, we had also subestimated the importance of requiring the specialization map to be quasicompact. 
	    Since the local models are proper, in the end, one ended up with the same concept. We clarify this below.
	    \end{remark}
    \begin{proposition}
	    \label{flatvsthick}
    Let $O$ be a complete discrete valuation ring over $\bbZ_p$ with uniformizer $\pi\in O$.
    Let $\calF$ be a prekimberlite with a formally adic map $\calF\to \Spd O$.
    Suppose that $\calF^{\on{an}}$ is a locally spatial diamond. The following statements hold.
    \begin{enumerate}
	    \item If $\calF\to \Spd O$ is proper, $\calF$ is $\pi$-flat (as in \cite[Definition 2.31]{AGLR22}) and formalizes geometric points, then $\calF$ is a thick spatial kimberlite.
	    \item If $\calF$ is a thick spatial kimberlite, then $\calF$ is $\pi$-flat.
    \end{enumerate}
    \end{proposition}
    \begin{proof}
	    Lets prove the first statement. 
	    We have a commutative diagram 
	    \begin{center}
	    \begin{tikzcd}
		    \calF \arrow{r} \arrow{d}{\on{SP}}  & \Spd O \arrow{d}{\on{SP}} \\
		    \calF^{\on{H}} \arrow{r} & \ast 
	    \end{tikzcd}
	    \end{center}
	    This shows that $\calF\to \ast$ is partially proper, and since $\calF^{\on{H}}\to \ast$ is weakly separated, by \Cref{standardpropertiesweaklypar}, $\calF\to \calF^{\on{H}}$ is partially proper, so it is a valuative prekimberlite. 
	    Now, $\calF^\red$ is proper over $\Spec O/\pi$ and $\calF^{\on{an}}$ is spatial, by \cite[Theorem 4.40]{Gle24} the map $\on{sp}:|\calF^{\on{an}}|\to |\calF^\red|$ is a spectral map of locally spectral spaces, but since both the target and the source are quasicompact the map is quasicompact. 
	    By \Cref{twodefinitionsofkimberlite}, $\calF$ is a kimberlite.

	    By definition of $\pi$-flatness, see \cite[Definition 2.31]{AGLR22}, we have a v-cover $\coprod_{i\in I}\Spd R^{+,\sharp}_i\to \calF$ where the map $\Spd R_i^{+,\sharp}\to \Spd O$ is formally adic, and $(R_i^\sharp,R_i^{+,\sharp})$ is a perfectoid Huber pair.  
	    Since $\calF$ is quasicompact, only finitely many $i\in I$ are needed and $\coprod_{i=1}^n \Spd R^{+,\sharp}_i\to \calF$ is a qcqs formally adic v-cover, so $\calF$ is a spatial kimberlite. 
	    It is thick since the specialization map is surjective for the $\Spd R^{+,\sharp}_i$, and consequently for $\calF$.  
	    This finishes the proof of the first statement.

	    For the second statement, assume that $\calF$ is a thick spatial kimberlite. 
	    Let $\Spa (R,R^+)\to \calF^{\on{an}}$ be a formalizable v-cover.
	    By \Cref{annoyinglemmaonsurjectivityofspecialization} and \Cref{qcqs-is-formal-adic}, $\Spd R^+\to \calF$ is a formally adic v-cover.  
	    Since the map $\calF\to \Spd O$ is assumed to be formally adic, then the map $\Spd R^+\to \Spd O$ is formally adic and given as the formalization of a map $\Spa(R^{\sharp},R^{\sharp,+})\to \Spa C$ for $R^\sharp$ an untilt of $R$ (possibly in characteristic $p$).
    \end{proof}

    Recall that there is a Beilinson--Drinfeld affine Grassmannian 
    \[\on{Gr}_\calG\to \Spd O_E\]
    which is an integral model v-sheaf of Scholze's $\on{B}_{dR}$-Grassmannian, and that there is a local model $\calM_{\calG,\mu}\to \Spd O_F$ which is an integral model of the Schubert variety defined by $\mu$ \cite[$\mathsection 4$]{AGLR22}.

    The following \Cref{Grassmanianiskimberlite} is a translation of the basic kimberlite-theoretic structural results in \cite{AGLR22}.
    \begin{theorem}[\cite{AGLR22}]
	    \label{Grassmanianiskimberlite}
    	With notation as above $\calM_{\calG,\mu}$, is a thick spatial kimberlite proper and formally adic over $\Spd O_F$. 
    \end{theorem}
    \begin{proof}
	    This follows directly from \Cref{flatvsthick} and \cite[Proposition 4.14]{AGLR22}.
    \end{proof}

   \section{The Henselian property for proper spatial kimberlites}
   We finish this section by showing that the condition of being a spatial kimberlite is useful and relevant to study the \'etale cohomology of a v-sheaf. 

   For this section, we let $k$ be a characteristic $p$ algebraically closed field and we work with v-sheaves over $\ast=\Spd(k,k)$.
   We fix a coefficient ring $\Lambda$ which is a torsion ring with $n\cdot \Lambda=0$ for $n$ relatively prime to $p$.
   We work with Scholze's $6$-functor formalism \cite{Sch17}. 
   We use derived notation i.e.~all of our functors are derived.

   Let $X$ be a spatial kimberlite, let $Y=X^{\on{an}}$, let $Z=X^\Red$ and denote by $j:Y\to X$ and $i:Z\to X$ the complementary open and closed immersions.
   We let $\pi_X:X\to \ast$ denote the structure morphism to $\ast$, similarly for $\pi_Y$ and $\pi_Z$.
   We assume that $X^\red$ is locally of perfectly finite presentation over $\Spec k$, and that it is perfectly proper over it. 
   We assume that $\on{dim.trg}\pi_Y<\infty$, and that $\on{sp}:|Y|\to |X^\red|$ is surjective (i.e. $X$ is a thick spatial kimberlite). 
   With the assumptions above, we show that $X$ satisfies two Henselian properties along $Z\to X$.
   \begin{lemma}
	   The map $\pi_X$ is representable in locally spatial diamonds, compactifiable and with $\on{dim.trg}\pi_X<\infty$ (i.e. $\pi_X$ is fdcs). 
	   In particular, 
	   \[\pi_{X,!}:\calD_{\acute{e}t}(X,\Lambda)\to \calD_{\acute{e}t}( \ast,\Lambda)\] 
	   exists and admits a right adjoint $\pi_X^!$. 
	   Moreover, $Y$ is a spatial diamond and $Y\to \ast$ is partially proper.
   \end{lemma}
   \begin{proof}
	   By \Cref{representabilityinlocallyspatialkimberlites}, the map is representable in locally spatial diamonds.
	   The map $\on{SP}_X:X\to X^{\on{H}}$ is partially proper since $X$ is valuative, and the map $X^{\on{H}}\to \ast$ is weakly partially proper by \Cref{valuativepreki} since $X^\red$ is assumed to be proper over $\Spec k$.
	   Since $\pi_X$ factors through these maps it is weakly partially proper, but it is also separated so it is partially proper.
	   By hypothesis $\on{dim.trg}\pi_Y<\infty$, and since $X^\red$ is locally of finite presentation, we also have that $\on{dim.trg}\pi_Z<\infty$. 
	   This is enough to conclude that $\on{dim.trg}\pi_X<\infty$.
	   The second claim is \cite[Definition 22.18]{Sch17} and \cite[Theorem 23.1]{Sch17}.

	   The third claim follows from the definition of spatial kimberlite since $Y=X^\an$. 
	   It is also clear that the inclusion $Y\to X$ is partially proper.
	   Indeed, any open subsheaf whose closed complement is representable by a v-sheaf must be partially proper.
   \end{proof}

   \begin{theorem}
	   \label{nearbycyclestheorem}
Let the context be as in the beginning of this section. 
Let $\calF\in \calD_{\acute{e}t}(Y,\Lambda)$ the following hold:
	   \begin{enumerate}
		   \item $\Gamma(X,j_!\calF)=0$. 
		   \item $\Gamma_c(X,j_*\calF)=0$.
	   \end{enumerate}
   \end{theorem}

   \begin{remark}
	   \label{remark work with lourenco}
	   In our work with Louren\c{c}o (see \cite{GL22}), we have already shown that the functor 
	   \[\Gamma(Y,j_!-):\calD_{\acute{e}t}(Y,\Lambda)\to \calD_{\acute{e}t}(\ast,\Lambda) \]
vanishes (i.e. it is isomorphic to the constant $0$ functor). 
   \end{remark}

   \begin{proof}
	   The first statement holds much more generally and it is the content of \cite[Lemma 4.3]{GL22} (see \Cref{remark work with lourenco}).
	   The second statement is more subtle. 
	   We wish to show that $\pi_{X,!}j_*:\calD_{\acute{e}t}(Y,\Lambda)\to \calD_{\acute{e}t}(\ast,\Lambda)$ is the $0$-functor. 

	   Let $S=\Spd k\rpot{t}$ and $s:\Spd k\rpot{t}\to \ast$ be the structure map, it is $\ell$-cohomologically smooth and $s^\ast$ is conservative. 
	   It suffices to show that $s^*\pi_{X,!}j_*$ is the $0$-functor. 
	   By smooth and proper base change, it suffices to show that $\pi_{s,!}\circ j_{s,*}\cong 0$ for the natural maps $j_s:Y\times S\subseteq X\times S$ and $\pi_s:X\times S\to S$.

	   Let $T\to Y$ be a universally open quasi-pro\'etale cover with $T=\Spa( R,R^+)$ a totally disconnected perfectoid space.
	   Fix a $\varpi\in R$ a pseudo-uniformizer. 
	   There is a continuous map 
	   \[\kappa:|(\Spd R^+\times \Spd k\pot{t})^{\on{an}}|\to [0,\infty]\]
	   measuring the relative value of $t$ against $\varpi$ and for every interval $I\subseteq [0,\infty]$ we let $U_{T,I}$ denote the open subset associated to the interior of $\kappa^{-1}(I)$.
	   For example, $T\times \Spd k\pot{t}=U_{T,[0,\infty)}$. 
	   Now, for all $I \subseteq (0,\infty)$ we let $U_I\subseteq Y\times S$ denote the open subset that is the image of $U_{T,I}\to Y\times S$.
	   We note that if $I\subseteq (0,\infty)$ is compact, then $U_I$ is a spatial diamond. 
	   Moreover, we claim that $|U_{[a,\infty)}|\cup |Z\times S|\subseteq |X\times S|$ defines a quasicompact open subset, and that in particular it gives rise to a spatial diamond that we will denote $U_{[a,\infty]}\subseteq X\times S$.
	   We prove our claim as follows, we have a map of spatial diamonds 
	   \[U_{T,[0,\infty]}=(\Spd R^+\times \Spd k\pot{t})^{\on{an}} \to (X\times \Spd k\pot{t})^{\on{an}}\]
	   coming from the formally adic map of spatial kimberlites 
	   \[\Spd R^+\times \Spd k\pot{t} \to X\times \Spd k\pot{t}.\] 
	   In general, 
	   \[f:|U_{T,[0,\infty]}| \to |(X\times \Spd k\pot{t})^{\on{an}}|\] 
	   might not be an open map, but it is still a quotient map.
	   The locus $|U_{[a,\infty)}|\cup |Z\times S|$ corresponds precisely to the image of $U_{T,[a,\infty]}$ which is quasicompact. 
	   Moreover, we have that $f^{-1}(f(U_{T,[a,\infty]}))=|U_{T,[a,\infty]}|\cup f^{-1}(f(U_{T,[a,\infty)})$ which is open. This implies that $U_{[a,\infty]}$ is open.

	   Let $\overline{U}_{[a,\infty]}$ denote the closure of $U_{[a,\infty]}$ inside of $U_{[0,\infty]}$, and let $k_{a,\infty}:U_{[a,\infty]}\to \overline{U}_{[a,\infty]}$ and $\overline{k}_{a,\infty}:\overline{U}_{[a,\infty]}\to U_{(0,\infty]}$ denote the open and respectively closed immersions. 
	   Analogously, we let $k_{a}:U_{[a,\infty)}\to \overline{U}_{[a,\infty)}$ and $\overline{k}_{a}:\overline{U}_{[a,\infty)}\to U_{(0,\infty)}$ denote the open and respectively closed immersions.
	   We also let $j_a:U_{[a,\infty)}\to U_{[a,\infty]}$ and $\overline{j}_a:\overline{U}_{[a,\infty)}\to \overline{U}_{[a,\infty]}$ denote the open immersion.

	   For the convenience of the reader we tabulate the maps involved.
	   \begin{enumerate}
		   \item  $k_{a,\infty}:U_{[a,\infty]}\to \overline{U}_{[a,\infty]}$ (open dense immersion, not partially proper).
		   \item $\overline{k}_{a,\infty}:\overline{U}_{[a,\infty]}\to U_{(0,\infty]}$ (closed immersion).
		   \item $k_{a}:U_{[a,\infty)}\to \overline{U}_{[a,\infty)}$ (open dense immersion, not partially proper).
		   \item $\overline{k}_{a}:\overline{U}_{[a,\infty)}\to U_{(0,\infty)}$ (closed immersion).
		   \item $j_a:U_{[a,\infty)}\to U_{[a,\infty]}$ (open dense immersion, partially proper).
		   \item $\overline{j}_a:\overline{U}_{[a,\infty)}\to \overline{U}_{[a,\infty]}$ (open dense immersion, partially proper).
		   \item $j_s: U_{(0,\infty)}\to  U_{(0,\infty]}$ (open dense immersion, partially proper).
		   \item $\pi_s: U_{(0,\infty]}\to  S$.
	   \end{enumerate}

	   For any $A\in \calD_{\acute{e}t}(X\times S,\Lambda)$, we may compute $\pi_{s,!}A \in \calD_{\acute{e}t}(S,\Lambda)$ by the formula 
	   \[\pi_{s,!}A\simeq \varinjlim \limits_{a \to 0}\pi_{s,*}\overline{k}_{a,\infty,*}k_{a,\infty,!}A_{|_{U_{[a,\infty]}}}.\]
	   Indeed, this follows from the fact that $X\times S\to S$ is a map of locally spatial diamond, the fact that the map $\pi_s\circ \overline{k}_{a,\infty}:\overline{U}_{[a,\infty]}\to S$ is proper, and the fact that the family $U_{[a,\infty]}$ is cofinal among quasicompact open subset of $U_{(0,\infty]}=X\times S$ \cite[Definition 22.13, Definition 22.4]{Sch17}.

	   When $A=j_{s,*}B$, we may rewrite this as 
	   \begin{align*}
		   \varinjlim \limits_{a \to 0}\pi_{s,*}\overline{k}_{a,\infty,*}k_{a,\infty,!} A_{|_{U_{[a,\infty]}}} 
		   &\cong \varinjlim \limits_{a \to 0}\pi_{s,*}\overline{k}_{a,\infty,*}k_{a,\infty,!}j_{a,*}B_{|_{U_{[a,\infty)}}} \\
		   &\cong \varinjlim \limits_{a \to 0}\pi_{s,*}\overline{k}_{a,\infty,*} \overline{j}_{a,*}k_{a,!}B_{|_{U_{[a,\infty)}}} \\
		   &\cong \varinjlim \limits_{a \to 0}\pi_{s,*} {j}_{s,*}\overline{k}_{a,*} k_{a,!}B_{|_{U_{[a,\infty)}}} \\
	   \end{align*}
	   The only subtle step of the computation is to justify that 
	   \[k_{a,\infty,!}j_{a,*}B_{|_{U_{[a,\infty)}}}\cong   \overline{j}_{a,*}k_{a,!}B_{|_{U_{[a,\infty)}}},\]
	   but this holds since $i_{a,s}^! k_{a,\infty,!}j_{a,*}B_{|_{U_{[a,\infty)}}}\cong 0$ where $i_{a,s}$ denotes the inclusion 
	   \[i_{a,s}:Z\times S\to \overline{U}_{[a,\infty]}\]
   and the term $\overline{j}_{a,*}k_{a,!}B_{|_{U_{[a,\infty)}}}$ is obtained from applying excision to get the triangle 
   \[0=i_{a,s,!}i_{a,s}^! k_{a,\infty,!}j_{a,*}B_{|_{U_{[a,\infty)}}}\to k_{a,\infty,!}j_{a,*}B_{|_{U_{[a,\infty)}}}\to \overline{j}_{a,*}k_{a,!}B_{|_{U_{[a,\infty)}}} .\]
	   
	   Finally, we can recognize that the expression 
	   \[\varinjlim \limits_{a \to 0}\pi_{s,*} \overline{j}_{s,*}\overline{k}_{a,*} k_{a,!}B_{|_{U_{[a,\infty)}}}\] 
	   computes the functors of \cite[Definition IV.5.2]{FS21}, which by \cite[Theorem IV.5.3]{FS21} vanish.
	   Indeed, $Y\to \ast$ is partially proper, $Y$ is a spatial diamond and $\on{dim.trg}\pi_Y<\infty$ which are the required hypothesis to apply \cite[Theorem IV.5.3]{FS21}.
   \end{proof}

    \section{The stack of shtukas.}
    \label{section on stack of shtukas}
    We finish this section with a more intricate and interesting example. 
    Recall the stack of $\calG$-shtukas, $\Sht_\calG\to \Spd O_E$, parametrizing the groupoid of triples 
    \[S=\Spa(R,R^+)\mapsto \{(S^\sharp,\calE,\Phi)\},\]
    where $S^\sharp$ is an isomorphism class of untilts of $S$ over $O_E$, $\calE$ is a $\calG$-torsor over $\calY_S$ and 
    \[\Phi: \varphi^*\calE\to \calE\]
    is an isomorphism defined away from the divisor $S^\sharp \hookrightarrow \calY_{S,[0,\infty)}$ that is meromorphic along this divisor (see \cite[Definition 5.3.5]{SW20}).
    Recall the local Hecke stack $\on{Hk}_\calG=[L^+\calG\backslash \Gr_\calG ]\to \Spd O_E$ (see \cite[\S 6]{FS21}), and its bounded version $\on{Hk}_{\calG,\mu}=[L^+\calG\backslash \calM_{\calG,\mu} ]\to \Spd O_F$, with 
    \[\on{Hk}_{\calG,\mu}\subseteq \on{Hk}_\calG\times_{\Spd O_E}\Spd O_F\]
    the only closed substack whose pullfack to $\Gr_\calG$ is $\calM_{\calG,\mu}$.

    Recall that we have a map 
    \[\Sht_\calG\to \on{Hk}_\calG\] and that we can pullback along the closed immersion 
    \[\on{Hk}_{\calG,\mu}\subseteq \on{Hk}_\calG\times_{\Spd O_E} \Spd O_F\] to obtained a closed substack 
    \[\Sht_{\calG,\mu}\subseteq \Sht_\calG\times_{\Spd O_E}\Spd O_F.\] 
    $\Sht_{\calG,\mu}$ is the moduli stack of $\mu$-bounded local $\calG$-shtukas.

   Further, recall that we have a map of small v-stacks 
   \[\sigma:\Sht_{\calG}\to \Bun_G,\]  
   that on functor of points, takes a triple $(S^\sharp,\calE,\Phi)$ to the only $G$-bundle on $X_{\on{FF},S}$ that agrees, as a $\varphi$-equivariant $G$-bundle, with $(\calE,\Phi)$ after pullback to $Y_{S,[r,\infty)}$ for a sufficiently large $r$ that avoids the untilt $S^\sharp\subseteq \calY_{S}$. 
     Recall that for all $b\in B(G)$ we have a map
     \[\calE_b:\Spd \overline{\bbF}_p\to \Bun_G,\]
     that gives rise to a locally closed substack $\Bun^b_G\subseteq \Bun_G$.
     We let 
     \[\Sht^b_{\calG,\mu}=\Sht_{\calG,\mu}\times_{\Bun_G} \Bun_G^b\]
     and 
     \[\Sht_{\calG,\mu}(b)=\Sht_{\calG,\mu}\times_{\Bun_G,\calE_b} \ast.\]
     By the definitions and by \cite[Proposition V.2.2]{FS21}, the map 
     \[\Sht_{\calG,\mu}(b)\to \Sht^b_{\calG,\mu}\] is a $\widetilde{G}_b$-torsor. 
     Proving the next theorem was our main motivation to introduce and study spatial kimberlites, and its proof will occupy the rest of the section.
    \begin{theorem}
	    \label{shtukasareactualllyspatial}
    The following hold.	
    \begin{enumerate}
	    \item The map $\Sht_{\calG,\mu}\to \Bun_G$ is representable in locally spatial diamonds. 
	    \item The stack $\Sht_{\calG,\mu}$ is an Artin v-stack.  
	    \item The v-sheaf $\Sht_{\calG,\mu}(b)$ is a locally spatial kimberlite.  
    \end{enumerate}
    \end{theorem}
    \begin{proof}
	    By \cite[Proposition IV.1.8.(iii)]{FS21} and \cite[Theorem IV.1.19]{FS21}, the first claim implies the second. 
	    Let us prove the first claim, we do this by reducing it to the third claim.
	    Let $W\calG$ (respectively $W^+\calG$) denote the functor classifying pairs
	    \[(R,R^+)\mapsto (R^\sharp, g)\]
	    where $R^\sharp$ is an isomorphism class of untilts given by $\xi\in \bbW(R^+)$ with $R^{\sharp,+}=\bbW(R^+)/\xi$ and $g\in \calG(\bbW(R^+)[\xi^{-1}])$ (respectively $g\in \calG(\bbW(R^+))$.
	    We have a v-surjective map $W\calG\to \on{Gr}_\calG$ that is a $W^+\calG$-torsor \cite{Ans18} \cite[Theorem 2.8]{Gle21}.
	    Let $W\calG_{\mu}= W\calG\times_{\on{Gr}_\calG} \calM_{\calG,\mu}$.
	    We claim that $W\calG_{\mu}$ is a spatial kimberlite.
	    It is clearly separated, and $\pi$-adic. 
	    This shows that it is formally separated. 
	    We can pick a formally adic v-cover $\Spd R^+\to  \calM_{\calG,\mu}$ over which the $W^+\calG$-torsor is trivial, this constructs for us a formally adic v-cover of the form \[\Spd R^+\times W^+\calG\to W\calG_{\mu}.\]
	    This proves that $W\calG_{\mu}$ is v-locally formal, so it is a specializing v-sheaf (see \Cref{specializing v-sheaf}).  
	    Also, $(W\calG_{\mu})^\red$ is a qcqs scheme, indeed it is a closed subscheme of the Witt-vector loop group. 
	    This shows that $W\calG_{\mu}$ is a prekimberlite (see \Cref{defiitioon f prekimbelrite}). 
	    Using \Cref{valuativityisappropriatelyv-local}, one shows that $W\calG_{\mu}$ is a valuative prekimberlite (see \Cref{valuative prekimberlite defi}).

	    Let $\Spa (R,R^+)\to (\calM_{\calG,\mu})^{\on{an}}$ a universally open cover by a strictly totally disconnected space.
	    Since $\calM_{\calG,\mu}$ formalizes points, by \Cref{prekimbproetaleformalizing} we get a map $\Spd R^+\to \calM_{\calG,\mu}$ which we may interpret as an untilted v-$\calG$-torsor over $\Spd \bbW(R^+)$ together with a trivialization away from $\xi_{R^\sharp}$. 
	    By the main theorem of \cite{güthge2024perfectprismaticfcrystalspadicshtukas}, this torsor is classical.
	    Since $\Spd R^+$ is totally disconnected, it must be trivial.
	    This shows that the $W^+\calG$-torsor over $\Spd R^+$ is trivial.
	    In particular, $\Spa( R,R^+)\times W^+\calG\to (W\calG_\mu)^{\on{an}}$ is a universally open, qcqs and pro\'etale map whose source is a spatial diamond, so the target is also a spatial diamond \cite[Proposition 11.24]{Sch17}. 
	    Moreover, $\Spd R^+\times W^+\calG\to W\calG_{\mu}$ is formally adic, qcqs, and surjective with source a spatial kimberlite. 
	    This implies that $W\calG_{\mu}$ is a spatial kimberlite (see \Cref{spatial-kimberlite-defi}).

	    Now that we have shown that $W\calG_{\mu}$ is a spatial kimberlite, we use this space as a v-surjective chart for the stack of shtukas.
	    We have an evident map 
	    \[f:W\calG_{\mu}\to \Sht_{\calG,\mu}\] 
	    that takes a $\calG$-matrix to the shtuka that such a matrix defines.
	    This map is v-surjective by \cite[Corollary 7.12, and by the proof of Theorem 7.13.(3)]{GI_Bunmer}.
     We claim that $f$ is qcqs. 

     Let us give the argument for quasicompactness, with the argument for quasiseparatedness being analogous. 
     Given a map $\Spa(R,R^+)\to \Sht_{\calG,\mu}$ with $\Spa( R,R^+)$ a product of points, we can lift it to a map $N:\Spa( R,R^+)\to W\calG_{\mu}$ and $\calF:=W\calG_\mu \times_{\Sht_{\calG,\mu}} \Spa( R,R^+)$ consists of those matrices in $M\in\calG(\bbW(R^\circ))$ such that 
     \[M^{-1}N\varphi(M)\in\calG(\bbW(R^+)[\xi^{-1}])\subseteq \calG(\bbW(R^\circ)[\xi^{-1}]).\]
     In other words, we have the following Cartesian diagram of sheaves over $\Spa( R,R^+)$,
    \begin{center}
    \begin{tikzcd}
	    \calF \arrow{rrr} \arrow{d}& &  & \arrow{d} W\calG_R \\
	    W^{+,\dagger}\calG_R    \arrow{rrr}{(-)^{-1}\cdot N\cdot \varphi(-) } & & & W^\dagger\calG_R.
    \end{tikzcd}
    \end{center}
    Here $W^\dagger\calG(A,A^+)=\calG(\bbW(A^\circ)[\frac{1}{\xi}])$ and $W^{+,\dagger}\calG(A,A^+)=\calG(\bbW(A^\circ))$.
    The map $W\calG\to W^\dagger\calG$ is a pro-open immersion along quasicompact transition maps and consequently quasicompact. 
    Moreover, after fixing Zariski closed immersions $\calG\to \on{GL}_n\to \on{M}_{n\times n} \times \on{M}_{n\times n}$ we may regard $W^{+,\dagger}\calG_R$ as a Zariski closed subsheaf of the functor $A\mapsto \on{M}_{n\times n}\times \on{M}_{n\times n}(\bbW(A^\circ))$ which is an infinite dimensional compact unit ball over $\Spa( R,R^+)$ and consequently quasicompact. This finishes showing that $\calF$ is quasicompact over $\Spa(R,R^+)$, and that the map $f:W\calG_{\mu}\to \Sht_{\calG,\mu}$ is qcqs as we wanted to show.

	    The map $W\calG_{\mu}\to \Bun_G$ is representable in locally spatial diamonds. 
	    Indeed, if $\Spa (R,R^+)$ is strictly totally disconnected, then by \Cref{representabilityinlocallyspatialkimberlites} $W\calG_{\mu}\times \Spa (R,R^+)$ is a locally spatial diamond and $W\calG_{\mu}\times_{ \Bun_G} \Spa(R,R^+)$ coincides with the pullback of the map \[W\calG_{\mu}\times \Spa( R,R^+)\to \Bun_G\times \Bun_G\] along the diagonal 
	    \[\Bun_G\xrightarrow{\Delta_{\Bun_G}} \Bun_G \times \Bun_G,\] which is representable in locally spatial diamonds, since $\Bun_G$ is an Artin v-stack (see \cite[Definition IV.1.1]{FS21}). 

	    We want to show that $Y=\Sht_{\calG,\mu}\times_{\Bun_G} \Spa( R,R^+)$ is a locally spatial diamond for all $\Spa (R,R^+)$ a strictly totally disconnected space. 
	    Let $T=\pi_0(\Spa( R,R^+))$ as a profinite topological space.
	    We have a continuous map $|Y|\to T$ that induces a map of v-sheaves $Y\to \underline{T}$.
	    Moreover, for every $t\in T$, $Y\times_{\underline{T}} \ast$ coincides with $\Sht_{\calG,\mu}\times_{\Bun_G} \Spa( C_t,C_t^+)$ where $\Spa( C_t,C_t^+)$ denotes the residue field at the closed point of the connected component $t\in T$.
	    Note that we have a v-surjective and qcqs map 
	    \[W\calG_{\mu}\times_{\Bun_G} \Spa( R,R^+)\to \Sht_{\calG,\mu}\times_{\Bun_G} \Spa( R,R^+)\]
	    whose source is a locally spatial diamond. 
	    By \cite[Lemma 13.5]{Sch17}, it suffices to show that each 
	    \[\Sht_{\calG,\mu}\times_{\Bun_G} \Spa (C_t,C_t^+),\]
	    for $t\in \pi_0(T)$, is a locally spatial diamond.
	    Now, we can use \Cref{representabilityinlocallyspatialkimberlites} to finish proving that the first statement is implied by the third statement. 
	    Indeed, 
	    \[\Sht_{\calG,\mu}\times_{\Bun_G} \Spa (C_t,C_t^+)\simeq (\Sht_{\calG,\mu}\times_{\Bun_G} \ast)\times_{\ast} \Spa (C_t,C_t^+),\]
and 
\[\Sht_{\calG,\mu}(b):=\Sht_{\calG,\mu}\times_{\Bun_G} \ast.\]

	    Finally, let us prove the third statement. 
	    Consider the map $(W\calG_{\mu}^\red)^\diamond\to \Bun_G$. 
	    Using \cite[Theorem 1.11]{GI_Bunmer}, this corresponds by adjunction to a map $W\calG_{\mu}^\red\to \ICG$, and for any $b\in B(G)$ we get a constructible locally closed subset $(W\calG_{\mu})_b^\red\subseteq (W\calG_{\mu})^\red$ (see \cite[Theorem 3.31.(1)]{zhu2025tamecategoricallocallanglands}).  
	    Let 
	    \[\WSht^b_{\calG,\mu}:=\widehat{W\calG_{\mu}}_{/(W\calG_{\mu})_b^\red}\] denote the formal neighborhood of $W\calG_{\mu}$ along $(W\calG_{\mu})_b^\red$ as in \Cref{formalnbbhoosoughtotbedefined}.  
	    This is the open subsheaf 
	    \[\WSht^b_{\calG,\mu}\subseteq W\calG_{\mu}\times_{\Bun_G}\Bun^b_G\] corresponding to the locus of points $x$ with Newton polygon $b_x=b$ and whose specialization $s=\on{sp}(x)\in W\calG_{\mu}^\red$ also has Newton polygon $b_s=b$.
The map $\WSht^b_{\calG,\mu}\to \Sht_{\calG,\mu}$ evidently factors through 
\[\Sht^b_{\calG,\mu}:=\Sht_{\calG,\mu}\times_{\Bun_G} \Bun_G^b.\]
Moreover, by \Cref{keytubularneighborhood}, $\WSht^b_{\calG,\mu}$ is a spatial kimberlite.  
We finish showing that $\Sht_{\calG,\mu}(b)$ is a locally spatial kimberlite through a sequence of lemmas.
\begin{lemma}
The $\varphi$-conjugation action of $W^+\calG$ on $\WSht^b_{\calG,\mu}$ realizes this space as a $W^+\calG$-torsor over $\Sht^b_{\calG,\mu}$.   
\end{lemma}
\begin{proof}
The map 
\[\WSht^b_{\calG,\mu}\to \Sht^b_{\calG,\mu}\]
is clearly $W^+\calG$-equivariant, we now show it is v-surjective.

Given $\alpha\in \Sht^b_{\calG,\mu}(R,R^+)$ with $(R,R^+)$ a product of points, we may find an isomorphism over $Y_{[r,\infty)}$ with $\calE_b$ and after fixing such isomorphism the $\calG$-torsor over $Y_{[0,\infty)}$ extends canonically to a $\calG$-torsor over $\Spec(\bbW(R^+))$, which is necessarily trivial.  
After fixing a trivialization, the data of the shtuka defines a matrix in $\WSht^b_{\calG,\mu}$ giving rise to the original shtuka. 
This takes care of the v-surjectivity part.

To show that the map is a $W^+\calG$-torsor, it suffices to show that if $\Spa(R,R^+)$ is a product of points and we are given two matrices, 
\[M_1,M_2\in\WSht^b_{\calG,\mu}(R,R^+),\]
	that define the same shtuka, then they are $W^+\calG$-$\varphi$-conjugate.
	If $R^+=\prod_{i\in I}C^+_i$, then $\calG(\bbW(R^+))=\prod_{i\in I}\calG(\bbW(C_i^+))$ so it suffices to do this pointwise.
	In other words, we may assume that $R^+=C^+$ for a $C^+$ a valuation ring with algebraically closed non-Archimedean field, and that 
	\[M_1,M_2\in \WSht^b_{\calG,\mu}(C,C^+).\]
	From Fargues' theorem (\cite[Theorem 1.12]{Fargues_courbe}), it follows that $M_1$ is $\varphi$-conjugate to $M_2$ by the action of some $N\in \calG(\bbW(O_{C}))$. 
Let $k$ be the residue field of $O_{C}$ and let $k^+\subseteq k$ be the valuation ring defining $C^+\subseteq O_{C}$, it suffices to $\varphi$-conjugate the residue matrices,
\[\overline{M_1},\overline{M_2}\in  (W\calG_{\mu})_b^\red(k^+),\]
by the action of some $\overline{N}_+\in \calG(\bbW(k^+))$. 
We claim that any two elements in $(W\calG_{\mu})_b^\red(k^+)$ whose induced points in $(W\calG_{\mu})_b^\red(k)$ are $\varphi$-conjugate have to be $\varphi$-conjugate already in $(W\calG_{\mu})_b^\red(k^+)$.
This boils down to the separatedness of affine Deligne--Lusztig varieties as we explain below.
Consider the stacky quotient of scheme-theoretic v-sheaves 
\[\Sht^{\on{sch}}_{\calG,\mu,b}:=\frac{(W\calG_{\mu})_b^\red}{\on{Ad}_{\varphi}(W^+\calG)^\red}.\]
This is the scheme-theoretic v-stack of $\mu$-bounded Witt vector shtukas whose Newton point is constant and equal to $b$.
We want to show that two $k^+$-points 
\[\underline{M_1}, \underline{M_2}\in \Sht^{\on{sch}}_{\calG,\mu,b}(k^+)\]
are isomorphic if and only if their induced $k$-points are isomorphic.
We have the following Cartesian diagram of schematic v-stacks
\begin{center}
\begin{tikzcd}
 X_{\calG,\mu}(b) \arrow{d} \ar{r} & \ast \ar{d} \\
  \Sht^{\on{sch}}_{\calG,\mu,b}\arrow{r} & \ast/\underline{G_b(E)},
\end{tikzcd}
\end{center}
where $X_{\calG,\mu}(b)$ is the affine Delinge--Lusztig variety attached to $(b,\mu)$.
From this diagram, and since $\Spec k$ and $\Spec k^+$ split every pro-\'etale cover, we obtain quotient formulas of groupoids 
\[\Sht^{\on{sch}}_{\calG,\mu,b}(k^+)=[X_{\calG,\mu}(b)(k^+)/\underline{G_b(E)}(k^+)]=[X_{\calG,\mu}(b)(k^+)/{G_b(E)}]\]
and similarly 
\[\Sht^{\on{sch}}_{\calG,\mu,b}(k)=[X_{\calG,\mu}(b)(k)/{G_b(E)}].\]
But, by separatedness of $X_{\calG,\mu}(b)$, we have a $G_b(E)$-equivariant injection\footnote{Since ADLV's are ind-proper this is actually a bijection.} 
\[X_{\calG,\mu}(b)(k^+)\to X_{\calG,\mu}(b)(k).\]
This finishes the proof.
\end{proof}
We will now consider the following diagram with Cartesian squares 
\begin{center}
\begin{tikzcd}
	\WSht_{\calG,\mu}(b) \arrow{r} \arrow{d}  & \Sht_{\calG,\mu}(b) \arrow{r} \arrow{d}  & \ast \arrow{d} \\
	\WSht^b_{\calG,\mu} \arrow{r} & \Sht^b_{\calG,\mu}\arrow{r} & \Bun_G^b\cong [\ast/\widetilde{G}_b]
\end{tikzcd}
\end{center}

\begin{lemma}
	\label{charts are locally spatial}
The space $\WSht_{\calG,\mu}(b)$ is a locally spatial kimberlite. 	
\end{lemma}
\begin{proof}
	By \cite[Lemma 2.31]{Gle21}, and since the map $\Sht_{\calG,\mu}(b)\to \ast$ is partially proper, $\Sht_{\calG,\mu}(b)$ is a valuative prekimberlite.
	Using \Cref{valuativityisappropriatelyv-local}, it is not hard to see that $\WSht_{\calG,\mu}(b)$ is valuative, since $\WSht_{\calG,\mu}(b)\to \Sht_{\calG,\mu}(b)$ is $W^+\calG$-torsor.  
	Since $(\WSht^b_{\calG,\mu})^{\on{an}}$ is a spatial diamond, we may find a universally open pro\'etale cover $\Spa( R,R^+)\to (\WSht^b_{\calG,\mu})^{\on{an}}$ and by \Cref{prekimbproetaleformalizing} a formalization $\Spd R^+\to \WSht^b_{\calG,\mu}$. 
	We claim that the map $\Spd R^+\to \Bun_G^b$ factors through $\ast$. 
	First we justify that $S=\Spa( R,R^+)\to \Bun_G^b$ factors through $\ast$.
	Since the map $[\ast/\underline{G_b(E)}]\to [\ast/\widetilde{G}_b]$ is formally smooth (see \cite[Definition IV.3.1, Proposition IV.3.5, Proposition IV.4.24]{FS21}), it admits \'etale local sections which are split by $S$.  
	Fixing a compact open subgroup $K\subseteq G_b(E)$ we get \'etale maps $[\ast/\underline{K}]\to [\ast/\underline{{G}_b}]$ along which we can also lift to obtain maps $S \to [\ast/\underline{K}]$. 
	We can express $\ast=\varprojlim_{K_n\subseteq K} [\ast/\underline{K_n}]$, as $K_n$ varies over the finite index subgroups of $K$.
	Since $S$ lifts to each of the $[\ast/\underline{K_n}]$ individually, and we can find compatible sections, it also lifts to the limit. 
	This finishes constructing a factorization $S\to \ast \to \Bun_G^b$.

	As in \cite[Remark 3.15]{GI_Bunmer}, we can interpret the map $\Spd R^+\to \Bun_G^b$ as $\varphi$-equivariant vector bundle over $Y_{S,(0,\infty]}$, and constructing a lift to $\ast$ is equivalent to constructing a $\varphi$-equivariant isomorphism with $\calE_b$. 
	Let $T=\Spa(R,R^\circ)$.
	By the reasoning above, we have a $\varphi$-equivariant isomorphism defined over $Y_{S,(0,\infty)}$, which extends canonically to an isomorphism over $Y_{T,(0,\infty]}$ by \cite[Proposition 2.1.3]{PR21}.

	We get the following commutative diagram of scheme theoretic v-sheaves.
	\begin{center}
	\begin{tikzcd}
	 \Spec R^\circ_\red \arrow{r} \arrow{d}  & \ast \arrow{d} \\
	 \Spec R^+_\red \arrow{r} & \ICG_b\cong [\ast/\underline{G_b(E)}]
	\end{tikzcd}
	\end{center}
	Arguing as above, we see that the map $\Spec R^+_\red\to [\ast/\underline{G_b(E)}]$ factors through $\ast$ since $\Spec R^+_\red$ splits every \'etale map. 
	This shows that $\Spec R^+_\red \times_{\ICG_b} \ast$ is a trivial $\underline{G_b(E)}$-torsor.
	Moreover, there is a unique trivialization restricting to the specified section $\Spec R^\circ_\red\to \Spec R^+_\red \times_{\ICG_b} \ast$.
	By \Cref{ring-computation}, our $\varphi$-equivariant trivialization extends to $Y_{S,(0,\infty]}$ which finishes showing that $\Spd(R^+)\to \Bun_G^b$ factors through $\ast$. 

	As a consequence of the above, we get an identification 
	\[\Spd R^+\times_{\Bun_G^b}\ast=\Spd R^+\times \widetilde{G}_b, \]
	which finishes showing that $\Spd R^+\times_{\Bun_G^b}\ast$ is a locally spatial kimberlite, by applying \Cref{stable-under-cartesian}.
	Since $\Spd R^+\times_{\Bun_G^b} \ast\to \WSht_{\calG,\mu}(b)$ is a qcqs formally adic v-cover, and it is universally open pro\'etale on analytic loci, then $\WSht_{\calG,\mu}(b)$ is also a locally spatial kimberlite. 
\end{proof}

\begin{lemma}
	\label{universallyopenn}
	The map $\WSht_{\calG,\mu}(b)\to \Sht_{\calG,\mu}(b)$ is universally open.	
\end{lemma}
\begin{proof}
	We prove that $\ast\to [\ast/W^+\calG]$ is universally open, this can be checked v-locally in the target. 
	In particular, it suffices to check that $W^+\calG\to \ast$ is universally open, or even better, that $W^+\calG_{R}\to \Spa(R,R^+)$ is open.
	We present $W^+\calG_R$ as the limit $W^+\calG_R=\varprojlim W^+\calG_R/W^{\leq n}\calG_R$, and each 
	\[\varprojlim W^+\calG_R/W^{\leq n}\calG_R\to \Spa(R,R^+)\]
	is $\ell$-cohomologically smooth and consequently an open mapping (see \cite[Proposition 23.11]{Sch17}). 
	The transition maps are qcqs, surjective and open maps of spatial diamonds, which allow us to conclude that $W^+\calG_R\to \Spa(R,R^+)$ is open (see \cite[Lemma 11.22]{Sch17}).
\end{proof}

By \cite[Lemma 2.31]{Gle21} and partial properness, $\Sht_{\calG,\mu}(b)$ is a valuative prekimberlite. 
Let $U=\Spec A\subseteq \Sht_{\calG,\mu}(b)^\red$ be an affine open subset we wish to show that $\widehat{\Sht_{\calG,\mu}(b)}_{/U}$ is an affine spatial kimberlite.  
Let $\tilde{U}\subseteq \WSht_{\calG,\mu}(b)^\red$ be the pullback of $U$. 
This is an affine subset of $\WSht_{\calG,\mu}(b)^\red$ so, by \Cref{charts are locally spatial}, $\widehat{\WSht_{\calG,\mu}(b)}_{/U}$ is a spatial kimberlite.
The map
\[ \widehat{\WSht_{\calG,\mu}(b)}_{/U}\to \widehat{\Sht_{\calG,\mu}(b)}_{/U}\]
is formally adic, v-surjective and qcqs, so if we knew that $\widehat{\Sht_{\calG,\mu}(b)}_{/U}$ was a kimberlite, it would follow from the definitions that $\widehat{\Sht_{\calG,\mu}(b)}_{/U}$ is a spatial kimberlite. 
Moreover, we already know that $\widehat{\Sht_{\calG,\mu}(b)}_{/U}$ is a valuative prekimberlite, so it suffices to show that $(\widehat{\Sht_{\calG,\mu}(b)}_{/U})^\an$ is a spatial diamond.
This is what we will argue.

Consider the map, 
\[(\widehat{\WSht_{\calG,\mu}(b)}_{/\tilde{U}})^\an\to (\widehat{\Sht_{\calG,\mu}(b)}_{/U})^\an\]
it is surjective, qcqs, and a universally open map.
Since the source of the map is a spatial diamond (see \Cref{universallyopenn}) and the target is a quasiseparated v-sheaf (see \Cref{quasiseparatednessofprekimberlites}), it follows that the target is a locally spatial v-sheaf (see \cite[Definition 12.12]{Sch17}).
To prove that $\Sht_{\calG,\mu}(b)^{\on{an}}$ is a spatial diamond it suffices to show, by \cite[Theorem 12.18]{Sch17}, that every point $x\in |\Sht_{\calG,\mu}(b)^{\on{an}}|$ admits a quasi-pro\'etale map from a geometric point $\Spa( C,C^+)$.
We do this in two cases.

For points over $x\in |\Sht_{\calG,\mu}(b)^{\on{an}}|$ lying over $\Spd E$, we may construct a geometric point over $x$ using the Grothendieck--Messing period map since on this loci $\Sht_{\calG,\mu}(b)$ admits an \'etale map to a $B_{\on{dR}}^+$-Grassmannian.

For the analytic points lying over $\Spd O_E/\pi$ we argue as follows.
Recall from \cite{GI_Bunmer} that there is a map $\Sht_{\calG,\overline{\bbF}_p}\to \Bun_G^\mer$.
We can stratify $\Sht_{\calG,\mu, \overline{\bbF}_p}(b)$ by the generic Newton polygon stratification (see \cite[Theorem 6.13]{GI_Bunmer}) and if $x\in |\Sht_{\calG,\mu,\overline{\bbF}_p}(b)^{\on{an}}|$, then its image in $\Bun_G^{\mer}$ lies in ${\calM}_{b_x}^\circ$ with $b_x\neq b$.
Since $\widetilde{\calM}_{b_x}^\circ$ is a spatial diamond (see \cite[Proposition V.3.6]{FS21}), we can construct a quasi-pro\'etale map $\Spa ( C,C^+)\to \Bun_G^\mer$ lying over the image of $x$ in $|\Bun_G^\mer|$.  
Finally, the fibers $\Spa(C,C^+)$-fibers of the map $\Sht_{\calG,\overline{\bbF}_p}\to \Bun_G^\mer$ are of the form $(\calF\ell^{\on{Witt}}_{\calG})^\diamond \times \Spd(C,C^+)$ (see \cite[Theorem 6.6.(2)]{GI_Bunmer}), and since this space is an ind-spatial diamond, all of the points $x\in |(\calF\ell^{\on{Witt}}_{\calG})^\diamond \times \Spd(C,C^+)|$ admit quasi-pro\'etale maps from a geometric point.
This finishes the proof that $\Sht_{\calG,\mu}(b)$ is a locally spatial kimberlite and the proof of \Cref{shtukasareactualllyspatial}. 
%
%
%
    \end{proof}

	\bibliography{biblio.bib}
	\bibliographystyle{alpha}
	
\end{document}